\documentclass[twoside, 11pt]{amsart}

\usepackage[english]{babel}

\usepackage[T1]{fontenc}
\usepackage[lf]{Baskervaldx} 
\usepackage[bigdelims,vvarbb]{newtxmath} 
\usepackage[cal=boondoxo]{mathalfa} 

\usepackage{caption}
\usepackage{latexsym}
\usepackage{tensor}
\usepackage{mathrsfs}
\usepackage{times,amsmath,amsthm}
\usepackage{graphicx}
\usepackage{a4wide}
\usepackage[all]{xy}
\usepackage{paralist}
\usepackage{subfigure}
\usepackage{multicol}
\usepackage{tikz}
\usepackage{tikz-cd}
\usetikzlibrary{decorations.pathmorphing,arrows,arrows.meta,calc,shapes,shapes.geometric,decorations.pathreplacing, fit, positioning, matrix}
\tikzcdset{arrow style=tikz, diagrams={>=stealth}}
\usepackage{comment}
\tikzset{
  optree/.style={scale=.5,thick,grow'=up,level distance=10mm,inner sep=1pt},
  comp/.style={draw=none,circle,fill,line width=0,inner sep=0pt},
  dot/.style={draw,circle,fill,inner sep=0pt,minimum width=3pt},
  circ/.style={draw,circle,inner sep=1pt,minimum width=4mm},
  emptycirc/.style={draw,circle,inner sep=1pt,minimum width=2mm},
  root/.style={level distance=10mm,inner sep=1pt},
  leaf/.style={draw=none,circle,fill,line width=0,inner sep=0pt},
  nodot/.style={draw,circle,inner sep=1pt},
}

\usepackage{color}
\definecolor{Chocolat}{rgb}{0.36, 0.2, 0.09}
\definecolor{BleuTresFonce}{rgb}{0.215, 0.215, 0.36}

\usepackage[colorlinks,final,hyperindex, ]{hyperref}
\hypersetup{citecolor=BleuTresFonce, linkcolor=Chocolat, urlcolor  = BleuTresFonce}
\usepackage[noabbrev,capitalize]{cleveref}

\usepackage[normalem]{ulem}

\usepackage{todonotes}

\let\oldtocsection=\tocsection
\let\oldtocsubsection=\tocsubsection

\renewcommand{\tocsection}[2]{\hspace{0em}\vspace{0.1em}\rule{0pt}{14pt}\oldtocsection{#1}{#2}\bf}
\renewcommand{\tocsubsection}[2]{\hspace{2em}\oldtocsubsection{#1}{#2}}

\def\k{\mathbb{k}}

\def\F{\mathrm{F}}
\def\d{\mathrm{d}}
\def\P{\mathcal{P}}
\def\C{\mathcal{C}}
\def\oC{\overline{\mathcal{C}}}
\def\1{\mathbb{1}}
\def\I{\mathcal{I}}
\def\id{\mathrm{id}}
\def\End{\mathrm{End}}
\def\eend{\mathrm{end}}
\def\gra{\mathrm{g}}
\newcommand{\Sy}{\mathbb{S}}
\def\dcGra{\mathsf{dsGra}}

\def\3ldcncGra{3\textsf{-}\mathsf{dsncGra}}

\newcommand{\ZZ}{\mathbb{Z}}

\newcommand{\ac}{\scriptstyle \text{\rm !`}}
\def\G{\mathfrak{G}}
\def\g{\mathfrak{g}}
\def\a{\mathfrak{a}}
\def\Hom{\mathrm{Hom}}

\def\cc{\circledcirc}

\def\MC{\mathrm{MC}}
\def\pap#1#2#3{#1\stackrel{#2}{\vcenter{\hbox{\text{\scalebox{1.2}{$\Join$}}}}}#3}
\def\ba{\bar{\alpha}}

\def\Gm{\mathcal{G}}
\def\H{\mathrm{H}}
\newcommand{\Frob}{\mathrm{Frob}}
\newcommand{\uFrob}{\mathrm{uFrob}}

\def\U{\mathcal{U}}

\def\II{
	\vcenter{\hbox{\begin{tikzpicture}[scale=0.1]
	\draw[thick] (0,0) -- (0,3);
	\end{tikzpicture}}}}

\def\BB{
	\vcenter{\hbox{\begin{tikzpicture}[scale=0.1]
	\draw[thick] (0,4) -- (2,2);
	\draw[thick] (4,4) -- (2,2);
	\draw[thick] (2,2) -- (2,0);
	\end{tikzpicture}}}}

\def\CC{
	\vcenter{\hbox{\begin{tikzpicture}[baseline=1.8ex,scale=0.1]
	\draw[thick] (2,4) -- (2,2);
	\draw[thick] (4,0) -- (2,2);
	\draw[thick] (2,2) -- (0,0);
	\end{tikzpicture}}}}

\def\DD{
	\vcenter{\hbox{\begin{tikzpicture}[baseline=1.8ex,scale=0.1]
	\draw[thick] (4,0) -- (0,4);
	\draw[thick] (4,4) -- (0,0);
	\end{tikzpicture}}}}

\def\EE{
	\vcenter{\hbox{\begin{tikzpicture}[baseline=1.8ex,scale=0.1]
	\draw[thick] (2,4) -- (2,2);
	\draw[thick] (4,0) -- (2,2);
	\draw[thick] (2,2) -- (0,0) -- (2,-2);
	\draw[thick] (2,-2) -- (4,0) -- (2,2);
	\draw[thick] (2,-2) -- (2,-4);	
	\end{tikzpicture}}}}
	
\def\FF{
	\vcenter{\hbox{\begin{tikzpicture}[scale=0.1]
	\draw[thick] (0,4) -- (2,2);
	\draw[thick] (4,4) -- (2,2) -- (2,4);
	\draw[thick] (2,2) -- (2,0);
	\end{tikzpicture}}}}
	
\def\GG{
	\vcenter{\hbox{\begin{tikzpicture}[baseline=1.8ex,scale=0.1]
	\draw[thick] (2,4) -- (2,2);
	\draw[thick] (4,0) -- (2,2);
	\draw[thick] (2,2) -- (0,0);
	\draw[thick] (2,2) -- (2,0);	
	\end{tikzpicture}}}}

\numberwithin{equation}{section}
\theoremstyle{plain}
\newtheorem{proposition}[equation]{Proposition}
\newtheorem{theorem}[equation]{Theorem}
\newtheorem{corollary}[equation]{Corollary}

\theoremstyle{definition}
\newtheorem{definition}[equation]{Definition}

\newtheorem{remark}[equation]{\sc Remark}



\title{Universal constructions in homotopical algebra}

\date{\today}

\author{Ricardo Campos}
\address{Universit\'e de Toulouse, Institut de Math\'ematiques de Toulouse, CNRS, UMR 5219,  Toulouse, France}
\email{ricardo.campos@math.univ-toulouse.fr}

\author{Bruno Vallette}
\address{Universit\'e Sorbonne Paris Nord, Laboratoire de G\'eom\'etrie, Analyse et Applications, CNRS, UMR 7539, Villetaneuse, France}
\email{vallette@math.univ-paris13.fr}

\thanks{
\newline
The authors are supported by ANR-20-CE40-0016 HighAGT}

\begin{document}

\maketitle

\begin{abstract}
We apply the effective integration theory of Lie-graph algebras, developed recently by the authors, 
to the deformation and homotopy theories of types of  bialgebras, that is structures controlled by a properad, like associative bialgebras, (involutive) Lie bialgebras, Frobenius bialgebras, double Poisson bialgebras, pre-Calabi--Yau algebras, quantum Airy structures, etc. 
In these cases, we provide their associated Deligne groupoid with an explicit homotopical description. 
We settle the Koszul hierarchy and the twisting procedure on the properadic level. 
We also give a conceptual construction of the homotopy transfer theorem in terms of gauge actions.
This work extends the formulas for the deformation theory of operadic algebras. 
\end{abstract}

\tableofcontents

\section*{Introduction}

\paragraph*{\bf Integration theory of Lie-graph algebras} 
The integration theory of complete Lie algebras into topological groups is performed via the intricate Baker--Campbell--Hausdorff (BCH) formula, see \cite{BF12}. When the Lie bracket comes from the skew-symmetrisation of a binary product, this latter one is called \emph{Lie-admissible}. Pre-Lie products are Lie-admissible operations satisfying an even stronger relation which allows one to consider exponential and logarithm isomorphisms mapping the Baker--Campbell--Hausdorff formula to a simpler and effective group structure. In the same way, the action of this \emph{gauge group} on Maurer--Cartan elements admits more efficient formulas in the pre-Lie case \cite{DotsenkoShadrinVallette16}. This integration theory for pre-Lie algebras admits powerful applications to the deformation theory of morphisms of operads, including homotopy algebra structures \cite{campos2019lie, DotsenkoShadrinVallette22, Emprin24}. 

\medskip 

We introduced in \cite{CV25I} an intermediate Lie-admissible algebraic structure bridging Lie algebras and pre-Lie algebras, termed \emph{Lie-graph algebras}. This structure comprises an infinite family of operations indexed by directed connected graphs
The leading example for such a new algebraic structure is the deformation complex of morphisms of properads. This includes types of bialgebras  that are structures equipped with operations having multiple inputs and multiple outputs. Such rich structures appear more and more nowadays in algebra, geometry, topology, and mathematical physics:
associative bialgebras, (involutive) Lie bialgebras \cite{CFL15}, (involutive) Frobenius bialgebras \cite{Yalin18}, 
double Poisson bialgebras \cite{VdB08, Leray19protoII, LV22}, pre-Calabi--Yau (bi)algebras \cite{TZ07Bis, KTV25, KTV23}, (balanced) infinitesimal bialgebras \cite{Aguiar00, Q23}, quantum Airy structures \cite{KS18, ABCO17, BV24}, etc. 
The integration theory of Lie-graph algebras is fully settled in \cite{CV25I}: graphical exponential and logarithms isomorphisms, effective gauge group, and effective action on Maurer--Cartan elements. 
The goal of the present paper is to apply these new fundamental results to  universal constructions of homotopy bialgebras. 

\medskip

{
\paragraph*{\bf Deligne groupoid} 
The largest range of homotopy bialgebras for which all the required nice properties 
($\infty$-morphisms, homotopy transfer theorem, etc.)
have been established so far are the ones given  by differential graded properads which are the cobar construction of a coproperad: $\mathcal P_\infty = \Omega \C$~. 
The deformation complex $\g_{\C,A}$ of gebra structures over a properad of this form on a chain complex $A$ admits a natural Lie-graph algebra structure \cite[Theorem~1.21]{CV25I}.
Applying the aforementioned general integration theory of Lie-graph algebras to this example, we are first able to 
describe explicitly the associated Deligne groupoid, which is made up of Maurer--Cartan elements related under the gauge group action.

\medskip

\textbf{Theorem \ref{thm:DeligneGroupoid}.}
\emph{
The Deligne groupoid associated to the Lie-graph algebra $\g_{\C,A}$ is given by $\Omega \C$-gebra structures on $A$ and $\infty$-morphisms with first component equal to the identity of $A$: 
\[
\mathrm{Deligne}(\g_{\C,A})\cong \left(\Omega \C\textrm{-}\,\mathrm{gebras}, \infty\textrm{-}\,\mathrm{isotopies}\right)\ .
\]
}


\paragraph*{\bf Universal constructions} 
In the operadic context, there exist three universal constructions of homotopy algebras: the homotopy transfer theorem \cite[Section~10.3]{LodayVallette12}, the Koszul hierarchy \cite{Markl15}, and the twisting procedure \cite{DotsenkoShadrinVallette22}. 
The first one, the \emph{homotopy transfer theorem} provides us with explicit formulas for the transport of homotopy gebra structures through homotopy equivalences without any loss of algebraic-homotopical data. 
It was recently extended to the level of properads in \cite{HLV19}. 
Generalising the operadic approach of \cite{DotsenkoShadrinVallette16}, we give here the homotopy transfer theorem for $\P_\infty$-gebras a conceptual interpretation in terms of the gauge group action (\cref{thm:HTT}), which emphasizes the perturbative aspect of this theory. 

\medskip

\paragraph*{\bf The Koszul hierarchy} 
In \cref{sec:hierarchy},  we generalise to properads the construction known as the \emph{Koszul hierarchy}, which originally is a method that 
produces a (shifted) $\mathrm{L}_\infty$-algebra measuring the failure of a square-zero operator to be a derivation of a product in a given commutative algebra \cite{Koszul85}. 
This construction is intimately related to the Koszul duality $\mathrm{Com}\text{-}\mathrm{Lie}$. 
Using the gauge group action, we show that a similar procedure can be applied to any quadratic properad $\P$ to construct a 
$\P_\infty$-gebra structure from the data of a $\P^!$-gebra structure, where $\P^!$ is the Koszul dual properad. 
In order to appreciate the range of novelty brought by this new fundamental construction, one has to keep in mind that, contrary to the operadic case, there exists \emph{no} bar-cobar construction between $\P$-gebras and $\P^!$-gebras for quadratic properads, for the simple reason that the notion of a free $\P$-gebra does not exist. So the present Koszul hierarchy provides us with the first universal construction of (homotopy) $\P$-gebras from $\P^!$-gebras. 
To illustrate this phenomenon and to show how effective this new theory is, we make explicit
the homotopy involutive Lie bialgebra structures obtained from Frobenius bialgebra structures in this way (\cref{thm:KoszulHierarchy}). 
In order to do so, we introduce a general method which can be used to compute efficiently any case. For the first time, this establishes a universal construction of \emph{symplectic field theories} from \emph{topological field theories}.

\medskip

\textbf{Theorem \ref{thm:KoszulHierarchy}.}
\emph{
The Koszul hierarchy associated to a chain complex $(A, \d)$ and a Frobenius bialgebra structure $(A, \mu, \Delta)$ is the shifted homotopy involutive Lie bialgebra structure given by 
\begin{align*}\label{eq:hierarchy}
\nu_{m,n}^g=&
\sum_{\substack{
g_1+g_2=g \\
k+l=n \\ 
\sigma\in \mathsf{Sh}(k, l)}} 
A_{m,l}^{g_2}\, 
\left(\left(\mu_{m, l+1}^{g_2} \circ_1^1 \d\right) \circ_1^1 \mu_{1, k}^{g_1}\right) \cdot \sigma\\
&+ 
\sum_{\substack{
k+l=n \\ 
\sigma\in \mathsf{Sh}(k, l)\\
\tau \in \mathsf{Sh}(1, m-1) }} 
(-1)^{k-1}A_{m-1,l}^{g}\, 
\tau^{-1}\cdot \left(\left(\mu_{m-1, l+1}^{g} \circ_1^1 \d\right) \circ_1^2 \mu_{2, k}^{0}\right) \cdot \sigma
\\ 
&+
C_{m,n}^{g}\, 
\left(\mu_{m, g+1}^{0} \circ_1^1 \d\right) \circ_{[g+1]}^{[g+1]} \mu_{g+1, n}^{0}~,
\end{align*}
for $n,m \geqslant 1$, $g\geqslant 0$, and $m+n+g \geqslant 3$, where the notation $\mathsf{Sh}(k, l)$ stands for the subset of $\Sy_n$ made up of the $(k, l)-$shuffles and where the last term satisfies $C^0_{m,n}=0$~. The coefficients $A$ and $C$ have explicit formulas.
}

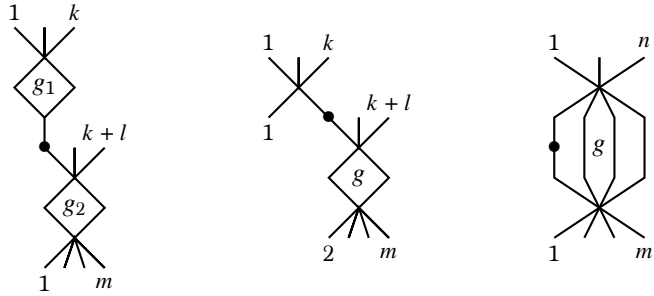
\begin{figure*}[h]
\begin{align*}
\vcenter{\hbox{\begin{tikzpicture}[scale=0.4]
	\draw[thick] (0,0)--(0,1);
	\draw[thick] (0,1)--(-1, 2)--(0,3)--(1,2)--cycle;
	\draw[thick] (-1,4)--(0,3)--(0,4)--(0,3)--(1,4);
	\draw[thick] (0,0)--(1,-1)--(1,0)--(1,-1)--(2,0);		
	\draw[thick] (1,-1)--(2,-2)--(1,-3)--(0,-2)--cycle;
	\draw[thick] (0,-4)--(1,-3)--(2,-4)--(1,-3)--(0.66,-4)--(1,-3)--(1.33,-4);
	\draw (0,0) node {{$\bullet$}} ; 	
	\draw (0,2) node {\scalebox{0.9}{$g_1$}} ; 	
	\draw (1,-2) node {\scalebox{0.9}{$g_2$}} ; 	
	\draw (-1,4.5) node {{\scalebox{0.8}{$1$}}} ; 				
	\draw (1,4.5) node {{\scalebox{0.8}{$k$}}} ; 					
	\draw (2,0.5) node {{\scalebox{0.8}{$k+l$}}} ; 						
	\draw (0,-4.5) node {{\scalebox{0.8}{$1$}}} ; 							
	\draw (2,-4.5) node {{\scalebox{0.8}{$m$}}} ; 								
\end{tikzpicture}}}
\qquad \qquad
\vcenter{\hbox{\begin{tikzpicture}[scale=0.4]
	\draw[thick] (0,0)--(-1,1)--(-2,0);
	\draw[thick] (-2,2)--(-1,1)--(-1,2)--(-1,1)--(0,2);
	\draw[thick] (0,0)--(1,-1)--(1,0)--(1,-1)--(2,0);		
	\draw[thick] (1,-1)--(2,-2)--(1,-3)--(0,-2)--cycle;
	\draw[thick] (0,-4)--(1,-3)--(2,-4)--(1,-3)--(0.66,-4)--(1,-3)--(1.33,-4);
	\draw (0,0) node {{$\bullet$}} ; 		
	\draw (1,-2) node {\scalebox{0.9}{$g$}} ; 	
	\draw (-2,2.5) node {{\scalebox{0.8}{$1$}}} ; 				
	\draw (0,2.5) node {{\scalebox{0.8}{$k$}}} ; 					
	\draw (2,0.5) node {{\scalebox{0.8}{$k+l$}}} ; 						
	\draw (-2,-0.5) node {{\scalebox{0.8}{$1$}}} ; 							
	\draw (0,-4.5) node {{\scalebox{0.8}{$2$}}} ; 							
	\draw (2,-4.5) node {{\scalebox{0.8}{$m$}}} ; 								
\end{tikzpicture}}}
\qquad \qquad
\vcenter{\hbox{\begin{tikzpicture}[scale=0.4]
	\draw[thick] (0,3)--(3,1)--(3,-1)--(0,-3);
	\draw[thick] (3,3)--(0,1)--(0,-1)--(3,-3);		
	\draw[thick] (1.5,3)--(1.5,2)--(1,1)--(1,-1)--(1.5,-2)--(1,-3);			
	\draw[thick] (1.5,3)--(1.5,2)--(2,1)--(2,-1)--(1.5,-2)--(2,-3);				
	\draw (0,0) node {{$\bullet$}} ; 	
	\draw (0,3.5) node {{\scalebox{0.8}{$1$}}} ; 				
	\draw (3,3.5) node {{\scalebox{0.8}{$n$}}} ; 					
	\draw (0,-3.5) node {{\scalebox{0.8}{$1$}}} ; 							
	\draw (3,-3.5) node {{\scalebox{0.8}{$m$}}} ; 	
	\draw (1.5,0) node {\scalebox{0.9}{$g$}} ; 									
\end{tikzpicture}}}
\end{align*}
\caption*{{\sc Figure.} \, The three types of composites appearing in the Koszul hierarchy of a Frobenius bialgebra.}
\end{figure*}

\medskip

\paragraph*{\bf The twisting procedure} 
In \cref{sec:twisting},  we generalise to the properadic level the \emph{twisting procedure}, which originally 
takes in the data of an $\mathrm{L}_\infty$-algebra and one of its Maurer--Cartan element and 
produces another $\mathrm{L}_\infty$-algebra structure, see \cite{DotsenkoShadrinVallette22}. Interpreted in terms of the Maurer--Cartan moduli space, this procedure amounts to localise at the given Maurer--Cartan element. From a given quadratic properad $\P$, we explain how to enlarge the deformation Lie-graph algebra 
$\g_{\P^!,A}$ controlling $\P_\infty$-gebras in order to produce a bigger gauge group including this time the elements of $A$. 
By definition, the gauge action of these elements is not stable on $\P_\infty$-gebras, but produces in general \emph{curved} $\P_\infty$-gebras. We coin the relevant general definition of \emph{$\P_\infty$-Maurer--Cartan equations} (\cref{def:PinftyMCequation}), which characterise the elements that produce back a $\P_\infty$-gebra under their gauge group action. As usual with the present approach, this procedure is effective and receives the following formula.

\medskip

\textbf{Theorem \ref{thm:TwProp}.}
\emph{
For any $\P_\infty$-gebra structure $\alpha$ on a complete graded module $A$, 
the gauge group action 
 \[(\1 -x)\cdot \alpha = \alpha \lhd(\1 +x)
 \] 
of a $\P_\infty$-Maurer--Cartan element $x$ produces a 
 $\P_\infty$-gebra structure~.
}

\medskip

This fundamental new form of the twisting procedure recovers all the known cases on particular algebras and bialgebra structures, like 
homotopy Lie bialgebras \cite{MW15}, homotopy involutive Lie bialgebras \cite{CFL15}, and pre-Calabi--Yau algebras \cite{LV22} for instance. Like the Koszul hierarchy, it is settled in such a way that it admits many degrees of freedom for a wider range of applications. For instance, it is shown that this twisting procedure can now accommodate not only the action of elements but also tensors of elements and more. 

\medskip

\paragraph*{\bf Range of applications} 
The fundamental results developed in the present paper actually apply to a wide range of recently studied properadic structures, some of them already mentioned above: 
associative bialgebras (using the bar construction for the cooperad $\C$ in order to get a model for their homotopy version), (involutive) Lie bialgebras, (involutive) Frobenius bialgebras, 
double Poisson bialgebras, pre-Calabi--Yau (bi)algebras, (balanced) infinitesimal bialgebras, and quantum Airy structures. 
This is why we expect that the present universal constructions will find even more applications in these domains in the future. 

\medskip

\paragraph*{\bf Conventions} 
For a uniform and light presentation, we work over a ground field $\k$ of characteristic $0$ though many results hold under weaker hypotheses. 
The underlying category is that of differential $\ZZ$-graded $\k$-vector spaces, using the homological degree convention, for which differentials have degree $-1$.  
The symmetric groups are denoted by $\Sy_n$.
We use the conventions of \cite{MerkulovVallette09I, MerkulovVallette09II, HLV19, CV25I} for  properads: for instance, we draw graphs with inputs at the top and outputs at the bottom. All the graphs present here will be directed by a global flow going from top to bottom.

\medskip

\paragraph*{\bf Acknowledgements} We would like to express our appreciation to 
Ga\"etan Borot, Fr\'ed\'eric Chapoton, Vladimir Dotsenko, Coline Emprin, Johan Leray, and Thomas Willwacher 
for interesting and helpful discussions. The second author is grateful to the IH\'ES for the long term invitation and the excellent working conditions. 

\section{Deligne groupoid of homotopy gebras}

Algebraic structures made up of operations with multiple inputs and multiple outputs can be encoded by the notion of a properad, which admits a Koszul duality theory \cite{Vallette07}. Bialgebras over properads $\Omega \C$ which are given by the cobar construction of a coproperad $\C$ have nice deformation and homotopic properties. 
First, they admit a convolution differential graded Lie algebra $\g_{\C, A}$ whose Maurer--Cartan elements are in one-to-one correspondence with $\Omega \C$-gebra structures on $A$, see \cite{MerkulovVallette09I, MerkulovVallette09II}.
Generalising the operadic case, a higher notion of \emph{$\infty$-morphism} was introduced recently in \cite[Section~3.3]{HLV19} for such homotopical properadic structures: $\infty$-endomorphisms are made up of a collection of maps which are packed into a map $\C \to \End_A$ satisfying some relation, see \cite[Proposition~3.16]{HLV19}. 
The \emph{$\infty$-isotopies} are the $\infty$-endomorphisms whose first component $\I \to \End_A$ is given by the identity $\id_A$~. 

\begin{proposition}\label{prop:MCInftyIso}
For any  conilpotent dg coproperad $\C$, two $\Omega \C$-gebra structures are gauge equivalent if and only if they are  $\infty$-isotopic. 
\end{proposition}

\begin{proof}
Recall from \cite[Proposition~1.27]{CV25I} that the dg Lie-admissible convolution algebra $\g_{\C, A}$ extends to a complete dg Lie-graph algebra satisfying $\g_{\mathcal{C}, A}= 
	\F_1 \g_{\mathcal{C}, A}$~, to which we can apply \cite[Theorem~2.29]{CV25I}. 
So two Maurer--Cartan elements $\ba, \bar{\beta} \in \MC\left(\g_{\C, A}\right)$ are gauge equivalent if and only if there exists $\lambda \in \Hom_{\Sy}\left(\oC, \End_A\right)_0$ such that 
\begin{equation*}
\bar{\beta}=
\pap{\exp(\lambda)}{\ba}{\exp(-\lambda)}-\left(\exp(\lambda); \d \big(\exp(\lambda)\big)\right)\cc \exp(-\lambda)~,
\end{equation*}
in $\g_{\C, A}$~. 
If we simply denote $\exp(\lambda)=\1+f$, with $f\in \Hom_{\Sy}\left(\oC, \End_A\right)_0$~, this relation becomes
\begin{equation*}
\bar{\beta}=
\pap{(\1 +f)}{\ba}{(\1 +f)^{-1}}-
\left((\1 +f); \partial(\1+f) \right)\cc (\1 +f)^{-1}
~, 
\end{equation*}

in $\g_{\C,A}$~.
The formula~$(\ast)$ given in the proof of \cite[Theorem~2.6]{CV25I} shows that 
\begin{equation*}
\pap{(\1 +f)}{\ba}{(\1 +f)^{-1}}=\left((\1 +f)\cc (\1+\ba)\cc (\1 +f)^{-1}\right)_{-1}~
\end{equation*}
since $f$ has degree $0$ and since $\ba$ has degree $-1$; using the same idea, one can also notice that 
\begin{equation*}
\left((\1 +f); \partial(\1+f) \right)\cc (\1 +f)^{-1}=\left(\big(\1 +f+\partial(\1+ f)\big)\cc (\1 +f)^{-1}\right)_{-1}~.
\end{equation*}
To summarize, we get 
\begin{equation*}
(\1+\bar{\beta})_{-1}=
\left((\1 +f)\cc (\1+\ba)\cc (\1 +f)^{-1}\right)_{-1}-\left(\big(\1 +f+\partial(\1 + f)\big)\cc (\1 +f)^{-1}\right)_{-1}~, 
\end{equation*}
in $\a_{\C,A}$~, which is equivalent to 
\begin{align*}
\left((\1+\bar{\beta})\cc (\1 +f)\right)_{-1}=&
\left((\1 +f)\cc (\1+\ba)\cc (\1 +f)^{-1}\cc (\1 +f)\right)_{-1}\\&
-\left(\big(\1 +f+\partial(\1+ f)\big)\cc (\1 +f)^{-1}\cc (\1 +f)\right)_{-1}~.
\end{align*}
The direct implication is straightforward, the converse is obtained similarly by composing on the right-hand side by 
$(\1 +f)^{-1}$~. 
This last equation is actually equal to 
\begin{equation*}
\bar{\beta}\lhd (\1 +f)=
(\1 +f)\rhd \ba-\partial(\1 +f)~, 
\end{equation*}
which is the equation in $\g_{\C,A}$ of an $\infty$-isotopy $\1+f$ between the two $\Omega \C$-gebra structures 
$\bar \alpha, \bar \beta \in \MC\left(\g_{\C, A}\right)$  by \cite[Proposition~3.16]{HLV19}. 
\end{proof}

\begin{theorem}\label{thm:DeligneGroupoid}
For any conilpotent dg coproperad $\C$
and any dg module $A$, the Deligne groupoid associated to the complete dg Lie algebra $\g_{\C, A}$ is isomorphic to the groupoid of 
$\Omega \C$-gebra structures on $A$ with their $\infty$-isotopies:
\[
\mathrm{Deligne}(\g_{\C,A})\cong \left(\Omega \C\textrm{-}\,\mathrm{gebras}, \infty\textrm{-}\,\mathrm{isotopies}\right)\ .
\]
\end{theorem}

\begin{proof}
This is a direct corollary of \cref{prop:MCInftyIso}.
\end{proof}

This general result admits many applications, for instance one for each Koszul properad $\P$, see \cite{Vallette07}: in this case, one considers the Koszul dual coproperad $\C=\P^{\ac}$. It shows that the Deligne groupoid associated to the convolution dg Lie algebra 
$\g_{\P^{\ac},A}$ encoding homotopy $\P$-gebra structures on $A$ is isomorphic to the groupoid of $\infty$-isotopies. 
For more details, we refer the reader 
to \cite[Section~6.1]{HLV19} for the case of homotopy Lie bialgebras, 
to \cite[Section~6.4]{HLV19} for the case of homotopy involutive Lie bialgebras (also known as $\mathrm{IBL}_\infty$-gebras), 
to \cite[Section~4]{LV22} and to \cite{KTV23} for the cases of  homotopy double Poisson gebras, $\mathrm{V}_\infty$-gebras, pre-Calabi--Yau algebras, 
and 
to \cite[Section~5]{Q23} for the case of homotopy balanced infinitesimal bialgebras. 
Quantum Airy structures introduced by Kontsevich--Soibelman in \cite{KS18}, see also \cite{ABCO17}, govern algebraically the 
Eynard--Orantin topological recursion \cite{EO07,EO09}. They are shown in \cite{BV24} to be encoded by a Koszul properad, so 
\cref{thm:DeligneGroupoid} applies to this case and describes the Deligne groupoid of their homotopical generalisation. 
The \emph{partition function }that computes the solutions to many problems in the enumerative geometry of complex curves and in low-dimensional quantum field theories  \cite{EICM,Ebook,B20}, 
is proved in \cite{BV24} to be given by a certain gauge element acting on homotopy 
quantum Airy structures 
 by the action formula of \cite[Theorem~2.29]{CV25I}. 
 Finally, considering the bar-cobar resolution, that is $\C=\mathrm{B}\,  \mathrm{Bialg}$, one can also treat the case of homotopy associative bialgebras. 
 
\begin{corollary}\label{cor:DeligneGroupoidHomotopyEquiv}
For any conilpotent dg coproperad $\C$
and any dg module $A$, the Deligne groupoid associated to the complete dg Lie algebra $\g_{\C, A}$ is isomorphic to the groupoid made up of morphisms of dg properads from $\Omega \C$ to $\End_A$ and homotopies between them:
\[
\mathrm{Deligne}(\g_{\C,A})\cong \left(\Hom_{\mathrm{dg\ properad}}\left(\Omega \C, \End_A\right), \sim_h\right)\ .
\]
\end{corollary}

\begin{proof}
We work in the model category structure on dg properads given in \cite[Appendix~A]{MerkulovVallette09II}. This defines the notion of a homotopy between morphisms of dg properads. 
For conilpotent dg coproperads $\C$, a cylinder object $\mathrm{Cyl}(\Omega \C)$ of the cobar construction $\Omega \C$ has been given in \cite[Proposition~1.12]{HLV24}, where it has been shown to coincide with the dg properad encoding two 
$\Omega \C$-gebra structures related by an $\infty$-isotopy.  \cref{thm:DeligneGroupoid} applies to conclude the proof. 

\end{proof}

\begin{remark}
This result extends \cite[Corollary~2]{DotsenkoShadrinVallette16} from the operad level to the properad level.
\end{remark}

\section{Homotopy transfer theorem}\label{sec:HTT}
The homotopy transfer theorem for gebras of type $\Omega \C$ has recently been settled in \cite[Section~4]{HLV19} with explicit formulas. 
The integration theory of Lie-graph algebras applied to the convolution algebra allows us to give the following interpretation to the homotopy transfer theorem in terms of the deformation gauge group action. 

\medskip

Recall that a \emph{contraction} of a chain complex $(A, d_A)$ is an endomorphism satisfying 
$h^2=0$ and $hdh=h$~, see \cite{EilenbergMacLane53}. 
Such a datum is equivalent to a deformation retract satisfying the side conditions \cite{BarnesLambe91}:
  \[
\begin{tikzcd}
(A,d_A)\arrow[r, shift left, "p"] \arrow[loop left, distance=1.5em,, "h"] & (H,d_H) \arrow[l, shift left, "i"]
\end{tikzcd} \ ,
\]
\begin{align*}
pi=\id_H\ , \quad ip -\id_A = d_A h+hd_A\ , \quad 
hi=0\ , \quad ph=0\ , \quad \text{and} \quad h^2=0 \ .
\end{align*}
For any positive integer $n$, we consider the following symmetric homotopies 
$$ h_n:=\frac{1}{n!}\sum_{\sigma \in \Sy_n}  \sum_{k=1}^n \left(
\id^{\otimes (k-1)} \otimes h \otimes \pi^{\otimes (n-k)}
\right)^\sigma$$
from $(ip)^{\otimes n}$ to $\id^{\otimes n}$. We denote generically by $\H$ the collection of homotopies $\{h_n\}_{n\geqslant 1}$.

\medskip
Given an $\Omega \C$-gebra structure on $A$, i.e. a Maurer--Cartan element $\ba$ in the convolution algebra $\g_{\C, A}$~, we consider the following elements of $\g_{\C, A}$~:
\[ \Upsilon \ \colon \ 
\oC\xrightarrow{\Delta_{\oC}} \Gm^c\big(\oC\big) \xrightarrow{\Gm^c(s\ba)}
\Gm^c\big(s\End_A\big) \xrightarrow{\mathrm{lev}}
\Gm_{\mathrm{lev}}^c\big(s\End_A\big) \xrightarrow{\id\mathrm{H\id}} \End_A\ ,\]
where $\Gm^c_{\mathrm{lev}}$ stands for the endofunctor on $\Sy$-bimodules made up levelled directed graphs \cite[Definition~4.2]{HLV19}, where $\mathrm{lev}$ is the levelisation morphism which amounts to putting one vertex on each level \cite[Definition~4.3]{HLV19}, and where the map ``$\id\mathrm{H}\id$'' amounts to removing all the suspensions $s$ and to labelling all the levels by the homotopy $\mathrm{H}$ except for the top level  and the bottom level which are labelled by the identity. It gives rise to the following two canonical elements of the deformation gauge group $\G$: 
\[
\Phi\coloneq \1 +\H\Upsilon \quad  \text{and} \quad \Psi\coloneq \1 + \Upsilon\H\ ,
\]
where left-hand  (resp. right-hand)  $\H$ means that we compose (resp. pre-compose) with $\H$, or equivalently we label the bottom (resp. top) level of $\Upsilon$ by $\H$. 

\begin{theorem}\label{thm:HTT}\leavevmode
\begin{enumerate}
\item The following two elements 
\[ \hat{\alpha}\coloneq \Phi^{-1}\cdot\bar{\alpha} \quad \text{and} \quad    \check{\alpha}\coloneq \Psi\cdot \bar{\alpha}\]
define $\Omega\C$-gebra structures on $A$.

\item The homotopy transferred $\Omega \C$-gebra structure $\bar\beta\coloneq p\Upsilon i$ on $H$~,
where the bottom level of $\Upsilon$ is labelled by $p$ and the top level by $i$~,  is equal to 
\[\bar \beta=p{\hat{\alpha}}i=p{\check{\alpha}}i \ .\]

\item The top and bottom squares of the following diagram of $\Omega \C$-gebras with $\infty$-morphisms are commutative: 
\[\begin{tikzcd}
\check{\alpha} \ar[r, "p"]         &      p\check{\alpha}i  \ar[d, equal]\ \\
\bar\alpha  \ar[u, rightsquigarrow, "\Psi"]    \ar[r, rightsquigarrow, "p_\infty",shift left=0.5ex]    &  \ar[l, rightsquigarrow, "i_\infty", shift left=0.5ex]   \bar\beta \ar[d, equal] \  \\
\hat{\alpha} \ar[u, rightsquigarrow, "\Phi"]         &    \ar[l, "i"]    p\hat{\alpha}i \ .
\end{tikzcd}\]
\end{enumerate}
\end{theorem}

\begin{proof}\leavevmode
\begin{enumerate}
\item This is a direct corollary of \cite[Theorem~2.29]{CV25I}.

\item {Recall first that $\bar\beta=p \Upsilon i$ was proved to be a $\Omega \C$-gebra structure on $H$ in \cite[Theorem~4.14]{HLV19}. Since it is made up of sums of levelled graphs with one vertex at each level, one can consider the bottom vertex of the given levelization each time.  Above it lies a globally levelled graph, possibly disconnected, with levels labelled by $\H$ and the top one by $i$. Lemma~4.13 of \cite{HLV19} shows that the same result is produced by $\bar\beta=p(\ba \lhd \Phi)i$~. Let us now prove the same formula for $p{\hat{\alpha}}i$. 
In $\g_{\C,A}$, the notation $p{\hat{\alpha}}i$ means that we first apply the map $\hat{\alpha}$ to $\oC$ in order to get elements of 
$\End_A$ that we compose with $p$ at each output and that we pullback by $i$ at each input. Under this notational convention, we get 
\[p{\hat{\alpha}}i= 
p\left(\pap{\Phi^{-1}}{\ba}{\Phi}\right)i
-p\left(\left(\Phi^{-1}; \d \left(\Phi^{-1}\right)\right)\cc \Phi\right)i\ ,\]
by \cref{thm:Action}. 
Proposition~\ref{prop:Inverse} shows that  the inverse element $\Phi^{-1}$ is equal $\1$ plus a sum of terms having at least one $h$ labelling one output. So the side condition $ph=0$ implies $p\Phi^{-1}=p$ and thus
\[p{\hat{\alpha}}i= 
p\left(\ba\lhd \Phi \right)i
-p\left(\d \left(\Phi^{-1}\right) \lhd \Phi\right)i\ .\]
Recall that 
$\d \left(\Phi^{-1}\right)=\partial_A \circ\Phi^{-1}- \Phi^{-1}\circ\d_\C$ by definition
and that $\d (\1)=0$ since $\d_\C(\I)=0$ and $\partial_A(\id_A)=0$~. 
Since $p$ is a chain map, we get 
\[p\left(\d \left(\Phi^{-1}\right)\right)=\partial_A \circ\left(p\Phi^{-1}\right)- \left(p\Phi^{-1}\right)\circ\d_\C~,\]
whose evaluation on $\oC$
 vanishes since it is the case for $p\Phi^{-1}=p$~.
In the end, we get $\bar\beta=p(\ba \lhd \Phi)i=p \hat{\alpha} i$, which concludes the proof. The other equality is proven in a similar way. 
}

\item The formulae for $\infty$-morphisms given in \cite[Sections~3-4]{HLV19} show that the composite of $\Psi$ with $p$ is equal to $p_\infty=p\circledcirc\Psi$ in $\a_{\C,A}$~. By \cref{prop:MCInftyIso} and by \cite[Theorem~4.14]{HLV19}, we know that $\Psi$ and $p_\infty$ are $\infty$-morphisms. Since $\Psi$ is an $\infty$-isotopy, we get 
$p_\infty\circledcirc \Psi^{-1}=p$ in $\a_{\C,A}$~ by \cite[Theorem~3.22]{HLV19}. We conclude that $p$ is an $\infty$-morphism by the fact that the composite of $\infty$-morphisms is an $\infty$-morphism. The arguments for $i$ and $i_\infty$ are similar. 
\end{enumerate}
\end{proof}

\begin{remark}
Contrary to the operadic case of \cite[Section~8]{DotsenkoShadrinVallette16}, we do not know so far whether the ``perturbed'' $\Omega \C$-gebra structures $\check{\alpha}$ and  $\hat{\alpha}$  are already restricted and corestricted to $H$. 
\end{remark}

\section{Koszul hierarchy}\label{sec:hierarchy} The main obstruction to extend the operadic results to the properadic level is the absence of a free $\P$-gebra functor. As a consequence, we have (so far) no bar-cobar adjunction between $\P$-gebras and $\P^!$-gebras, which in the Koszul duality case $\mathrm{Com}^!\cong \mathrm{sLie}$ provides us with the functors of rational homotopy theory of Quillen \cite{Quillen69} and Sullivan \cite{Sullivan77} for instance.
An important but different construction between algebras over Koszul dual operads is given by the \emph{Koszul hierarchy}. 
Originally due to J.-L. Koszul \cite{Koszul85} and developed by M. Markl in \cite{Markl15,Markl14}, the Koszul hierarchy is a method that takes as input a (graded) commutative algebra $A$ together with a (degree 1) square-zero operator $\Delta$, and constructs a (shifted) $\mathrm{L}_\infty$-algebra on $A$ such that the $n$-bracket vanishes if and only if $\Delta$ is a differential operator of order $n-1$.
The Koszul hierarchy was interpreted conceptually with the pre-Lie deformation theory in \cite[Section~5]{DotsenkoShadrinVallette16} and thus extended to operads beyond the cases of commutative and associative algebras. 
We now develop it further to the properadic level using the Lie-graph gauge action settled in the previous sections. This introduces for the first time a universal way to construct  $\P_\infty$-gebra structures  from $\P^!$-gebra structures, on the same underlying space.

\medskip

Let $\P$ be a quadratic properad and let $(A, \d)$ be a chain complex viewed as an abelian $\P$-gebra structure.
Notice that the presence of the differential prevents it from being completely trivial. 
Let $\theta : \P^! \to \End_A$ be a $\P^!$-gebra structure on $A$. 
Suppose that we have a (iso)morphism of $\Sy$-bimodules 
\[\mathrm{K} : {\P}^{\ac} \to  \P^!\]
from the Koszul dual coproperad to the Koszul dual properad, which sends the counit to the unit, for instance via a chosen basis. 
Such a data gives rise to the following element of the deformation gauge group associated to the convolution algebra $\g_{{\mathcal P^{\ac}}, A}$:
\[\Theta\coloneq  \theta(\mathrm{K}) \in  \G_{{\mathcal P^{\ac}}, A}\ .\]

\begin{definition}[Koszul hierarchy]
The \emph{Koszul hierarchy} associated to the data $(A\, , \d\, , \theta\, , \mathrm{K})$ is the $\P_\infty$-gebra structure on the chain complex $(A, \d)$ given by the action of $\Theta^{-1}$ on $\d$:
\[\Theta^{-1} \cdot \d =  -\left(\Theta^{-1}; \d\left( \Theta^{-1} \right)\right)\cc\Theta=
\Theta^{-1} \cc \left(\Theta; \d \Theta \right) \in \MC\left(\g_{{\mathcal P^{\ac}}, A}\right)\ . \]
\end{definition}

The Koszul hierarchy provides us with a well-defined $\P_\infty$-gebra structure by \cite[Theorem~2.29]{CV25I}.

\begin{remark}
Notice that the Koszul hierarchy retains almost no \emph{homotopical} information since it produces a $\P_\infty$-gebra structure which is gauge trivial by definition.
Equivalently, the Koszul hierarchy is $\infty$-isotopic to the underlying chain complex with trivial $\P$-structure 
by \cref{thm:DeligneGroupoid}. 
\end{remark}

\medskip

Let us make explicit the following noteworthy new example: let $\P=\mathrm{sIBiLie}$ be the properad encoding shifted involutive Lie bialgebras, that is with bracket and cobracket both of degree $-1$, see \cite{CFL15}. 
Its Koszul dual properad $\P^!\cong\Frob$ is the properad encoding Frobenius bialgebras by \cite{Vallette07} and both properads  are Koszul by \cite{CMW16}.  The properad 
$\P^!\cong\Frob$ is one-dimensional in each arity and genus; we denote its canonical basis by 
$\left\{p_{m,n}^g\right\}_{m,n \geqslant 1, g\geqslant 0}$~. The composite of 
$p_{m_1,n_1}^{g_1}$ with $p_{m_2,n_2}^{g_2}$
along $k$ edges 
is equal to \[p_{m_1+m_2-k,n_1+n_2-k}^{g_1+g_2+k-1}~,\]
 see \cite[Proposition~6.7]{HLV19}. The linear dual basis 
$\left\{c_{m,n}^g\coloneq \left(p_{m,n}^g\right)^*\right\}_{m,n \geqslant 1, g\geqslant 0}$ provides us with a basis for  the coproperad $\P^{\ac}\cong\Frob^*$~. 
So the associated convolution algebra is isomorphic to 
\[\g_{{\mathcal \Frob^*}, A}\cong 
\prod_{\substack{m,n \geqslant 1 \\ g\geqslant 0 \\ m+n+g \geqslant 3}} \Hom\left(A^{\odot n}, A^{\odot m}\right) \hbar^g~, \]
where $A^{\odot n}\coloneq A^{\otimes n}/\Sy_n$ is the symmetric tensor product, 
with the Lie-admissible product given by 
\[f_2\hbar^{g_2}\star f_1\hbar^{g_1}=\sum_{k=1}^{\mathrm{min}(m_1,n_2)}
\sum_{\substack{I\subseteq [m_1]\\ J \subseteq [n_2] \\ |I|=|J|=k}} \left(f_1 \circ_J^I f_2\right) \hbar^{g_1+g_2+k-1}~.\]

\medskip 

We consider here the canonical isomorphism of $\Sy$-bimodules 
$\mathrm{K} : \Frob^* \to  \Frob$ that send each basis element $c_{m,n}^g$ to its linear dual $p_{m,n}^{g}$~. 
Let  $(A, \mu, \Delta)$ be a Frobenius bialgebra structure, where 
$\mu \colon A^{\odot 2} \to A$ is an associative commutative product and where $\Delta \colon A \to A^{\odot 2}$ is a compatible 
coassociative cocommutative coproduct. The induced element $\theta(\mathrm{K})$ of the deformation gauge group is given by 
\begin{center}
\begin{tikzcd}[column sep=large]
p_{m,n}^g \arrow[r,"\theta"]  &  \mu_{m,n}^g\coloneq \Delta^{m-1}\circ \mu^g \circ \Delta^g \circ \mu^{n-1}~, \\
c_{m,n}^g  \arrow[u,"\mathrm{K}"] \arrow[ur,"\Theta"'] & 
\end{tikzcd}
\end{center}
where $\mu^k$ stands for the composite of $k$ products $\mu$ and where 
$\Delta^k$ stands for the composite of $k$ coproducts $\Delta$. By convention $\mu^0=\Delta^0=\id_A$~.

\medskip

Under the isomorphism 
\[\G_{\Frob^*, A}\cong \prod_{\substack{n,m \geqslant 1 \\  g\geqslant 0}} \Hom\left(A^{\odot n}, A^{\odot m}\right) \hbar^g~,\] 
this element of the deformation gauge group is equal to 
\[\Theta=
\sum_{\substack{m,n \geqslant 1 \\  g\geqslant 0}}  \mu_{m,n}^g \, \hbar^g=
\id_A+\mu+\Delta+ \mu^2+\Delta^2+ \Delta \mu+\mu\Delta\hbar+\cdots=
\II+ \BB + \CC +\FF+\GG+ \DD+\EE\,\hbar+\cdots
\ .\]

\begin{theorem}\label{thm:KoszulHierarchy}
Let $(A, \d)$ be a chain complex whose underling graded module is equipped with a Frobenius bialgebra structure $(A, \mu, \Delta)$. The shifted homotopy involutive Lie bialgebra structure 
\[\nu \coloneq 
\Theta^{-1} \cdot \d
\in \MC\left(\g_{{\mathcal \Frob^*}, A}\right)\]
obtained by the Koszul hierarchy associated to these data is equal to 
\begin{align*}\label{eq:hierarchy}
\nu_{m,n}^g=&
\sum_{\substack{
g_1+g_2=g \\
k+l=n \\ 
\sigma\in \mathsf{Sh}(k, l)}} 
A_{m,l}^{g_2}\, 
\left(\left(\mu_{m, l+1}^{g_2} \circ_1^1 \d\right) \circ_1^1 \mu_{1, k}^{g_1}\right) \cdot \sigma\\
&+ 
\sum_{\substack{
k+l=n \\ 
\sigma\in \mathsf{Sh}(k, l)\\
\tau \in \mathsf{Sh}(1, m-1) }} 
(-1)^{k-1}A_{m-1,l}^{g}\, 
\tau^{-1}\cdot \left(\left(\mu_{m-1, l+1}^{g} \circ_1^1 \d\right) \circ_1^2 \mu_{2, k}^{0}\right) \cdot \sigma
\\ 
&+
C_{m,n}^{g}\, 
\left(\mu_{m, g+1}^{0} \circ_1^1 \d\right) \circ_{[g+1]}^{[g+1]} \mu_{g+1, n}^{0}~,
\end{align*}
for $n,m \geqslant 1$, $g\geqslant 0$, and $m+n+g \geqslant 3$, where the notation $\mathsf{Sh}(k, l)$ stands for the subset of $\Sy_n$ made up of the $(k, l)-$shuffles and where the last term satisfies $C^0_{m,n}=0$~. The coefficients $A$ are given by 
\begin{align*}
A^{0}_{1,0}=1 \qquad \text{and} \qquad 
A^g_{m,n}=\sum_{k=1}^{m+n+2g-1} (-1)^k
\sum_{\Xi_k(m,n,g)} \prod_{j=2}^k \mathrm{T}(p_j, |I_j|, g_j, p_{j-1})~, 
 \end{align*}
for any $m\geqslant 1$, $n\geqslant 0$ and $g\geqslant 0$ with $m+n+2g\geqslant 2$~, 
where the sums run over the sets $\Xi_k(m,n,g)$ made up of the partitions $I_1 \sqcup I_2 \sqcup \cdots \sqcup I_k=\{2, \ldots, n\}$, the integer partitions $g_1+g_2+\cdots+g_k=g$, and the integers $p_1, \ldots, p_{k-1}$ such that 
$1\leqslant p_1 \leqslant p_2+g_2 \leqslant p_3 + g_2+g_3 \leqslant p_{k-1}+g_2+\cdots+g_{k-1}\leqslant 
m+g_2+\cdots+g_k$, satisfying $p_{j}-p_{j-1}+|I_j|+2g_j\geqslant 1$, for any $1\leqslant j\leqslant k$, under the convention $p_0=1$ and $p_k=m$, and 
where \[
\mathrm{T}(m, i, g,p)\coloneq\sum_{k=\max(1, p-g)}^p \mathrm{S}(m,k)k^i \binom{p-1}{k-1}\binom{g-p+2k-1}{k-1}~,
\]
with $\mathrm{S}(m,k)$ the Stirling number of second kind.
The coefficients $C$ are given by 
\begin{align*}
C^{1}_{1,1}=-1 \qquad \text{and} \qquad 
C^g_{m,n}= 
\sum_{k=1}^{m+n+2g-4} (-1)^{k+1}
\sum_{\Xi_k(m,n,g)} \left(n+(-1)^{p_1+|I_1|}\right)\prod_{j=2}^k \mathrm{T}(p_j, |I_j|, g_j, p_{j-1})~,
\end{align*}
for any $m,n,g\geqslant 1$ with $m+n+2g\geqslant 2$.
\end{theorem}
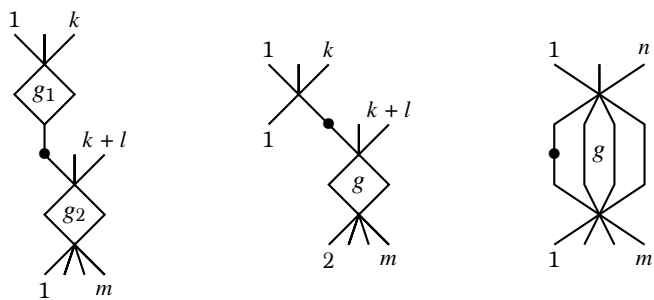
\begin{figure*}[h]
\begin{align*}
\vcenter{\hbox{\begin{tikzpicture}[scale=0.4]
	\draw[thick] (0,0)--(0,1);
	\draw[thick] (0,1)--(-1, 2)--(0,3)--(1,2)--cycle;
	\draw[thick] (-1,4)--(0,3)--(0,4)--(0,3)--(1,4);
	\draw[thick] (0,0)--(1,-1)--(1,0)--(1,-1)--(2,0);		
	\draw[thick] (1,-1)--(2,-2)--(1,-3)--(0,-2)--cycle;
	\draw[thick] (0,-4)--(1,-3)--(2,-4)--(1,-3)--(0.66,-4)--(1,-3)--(1.33,-4);
	\draw (0,0) node {{$\bullet$}} ; 	
	\draw (0,2) node {\scalebox{0.9}{$g_1$}} ; 	
	\draw (1,-2) node {\scalebox{0.9}{$g_2$}} ; 	
	\draw (-1,4.5) node {{\scalebox{0.8}{$1$}}} ; 				
	\draw (1,4.5) node {{\scalebox{0.8}{$k$}}} ; 					
	\draw (2,0.5) node {{\scalebox{0.8}{$k+l$}}} ; 						
	\draw (0,-4.5) node {{\scalebox{0.8}{$1$}}} ; 							
	\draw (2,-4.5) node {{\scalebox{0.8}{$m$}}} ; 								
\end{tikzpicture}}}
\qquad \qquad
\vcenter{\hbox{\begin{tikzpicture}[scale=0.4]
	\draw[thick] (0,0)--(-1,1)--(-2,0);
	\draw[thick] (-2,2)--(-1,1)--(-1,2)--(-1,1)--(0,2);
	\draw[thick] (0,0)--(1,-1)--(1,0)--(1,-1)--(2,0);		
	\draw[thick] (1,-1)--(2,-2)--(1,-3)--(0,-2)--cycle;
	\draw[thick] (0,-4)--(1,-3)--(2,-4)--(1,-3)--(0.66,-4)--(1,-3)--(1.33,-4);
	\draw (0,0) node {{$\bullet$}} ; 		
	\draw (1,-2) node {\scalebox{0.9}{$g$}} ; 	
	\draw (-2,2.5) node {{\scalebox{0.8}{$1$}}} ; 				
	\draw (0,2.5) node {{\scalebox{0.8}{$k$}}} ; 					
	\draw (2,0.5) node {{\scalebox{0.8}{$k+l$}}} ; 						
	\draw (-2,-0.5) node {{\scalebox{0.8}{$1$}}} ; 							
	\draw (0,-4.5) node {{\scalebox{0.8}{$2$}}} ; 							
	\draw (2,-4.5) node {{\scalebox{0.8}{$m$}}} ; 								
\end{tikzpicture}}}
\qquad \qquad
\vcenter{\hbox{\begin{tikzpicture}[scale=0.4]
	\draw[thick] (0,3)--(3,1)--(3,-1)--(0,-3);
	\draw[thick] (3,3)--(0,1)--(0,-1)--(3,-3);		
	\draw[thick] (1.5,3)--(1.5,2)--(1,1)--(1,-1)--(1.5,-2)--(1,-3);			
	\draw[thick] (1.5,3)--(1.5,2)--(2,1)--(2,-1)--(1.5,-2)--(2,-3);				
	\draw (0,0) node {{$\bullet$}} ; 	
	\draw (0,3.5) node {{\scalebox{0.8}{$1$}}} ; 				
	\draw (3,3.5) node {{\scalebox{0.8}{$n$}}} ; 					
	\draw (0,-3.5) node {{\scalebox{0.8}{$1$}}} ; 							
	\draw (3,-3.5) node {{\scalebox{0.8}{$m$}}} ; 	
	\draw (1.5,0) node {\scalebox{0.9}{$g$}} ; 									
\end{tikzpicture}}}
\end{align*}
\caption{The three types of composites A,B, and C appearing in the Koszul hierarchy of a Frobenius bialgebra.}
		\label{Fig:ABC}
	\end{figure*}

\begin{proof}
First, we notice that the Koszul hierarchy $\nu=\Theta^{-1}\cdot \d$ is the unique degree $-1$ element in $\g_{{\mathcal \Frob^*}, A}$ satisfying 
\begin{equation}\label{eq:CRUCIAL}\tag{$\divideontimes$}
\partial \Theta = \Theta \rhd \nu~.
\end{equation}
 Indeed \cref{prop:MCInftyIso} shows that $\Theta$ defines an $\infty$-isotopy from the abelian 
shifted homotopy involutive Lie bialgebra structure $(A,\d)$ to $\nu$ and so it satisfies Equation~\eqref{eq:CRUCIAL}. In the other way round, we use the fact that the deformation gauge group acts on the set of degree $-1$ elements of $\g_{{\mathcal \Frob^*}, A}$ under the same formula given by \cite[Theorem~2.29]{CV25I}. 
The proof of \cref{prop:MCInftyIso} shows that Equation~\eqref{eq:CRUCIAL} is equivalent to 
$\Theta\cdot \nu = \d$~, for any degree $-1$ element $\nu$ of $\g_{{\mathcal \Frob^*}, A}$~. Therefore, we get 
$\nu =\Theta^{-1}\cdot \d$, which is a Maurer--Cartan element since $\d$ is. 

\medskip

Let us call respectively by A and C the first of the third types of composites that appear in the formula of $\nu$, 
see \cref{Fig:ABC} where the bullet represents the differential $\d$. The second type of composites belongs to the following more general type B: 
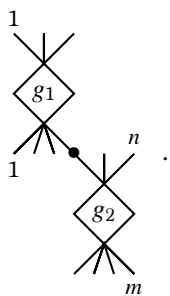
\[
\vcenter{\hbox{\begin{tikzpicture}[scale=0.4]
	\draw[thick] (0,0)--(-1,1)--(-2,0)--(-1,1)--(-0.66,0)--(-1,1)--(-1.33,0);
	\draw[thick] (-1,1)--(-2, 2)--(-1,3)--(0,2)--cycle;
	\draw[thick] (-2,4)--(-1,3)--(-1,4)--(-1,3)--(0,4);
	\draw[thick] (0,0)--(1,-1)--(1,0)--(1,-1)--(2,0);		
	\draw[thick] (1,-1)--(2,-2)--(1,-3)--(0,-2)--cycle;
	\draw[thick] (0,-4)--(1,-3)--(2,-4)--(1,-3)--(0.66,-4)--(1,-3)--(1.33,-4);
	\draw (0,0) node {{$\bullet$}} ; 	
	\draw (-1,2) node {\scalebox{0.9}{$g_1$}} ; 	
	\draw (1,-2) node {\scalebox{0.9}{$g_2$}} ; 	
	\draw (-2,4.5) node {{\scalebox{0.8}{$1$}}} ; 									
	\draw (2,0.5) node {{\scalebox{0.8}{$n$}}} ; 						
	\draw (-2,-0.5) node {{\scalebox{0.8}{$1$}}} ; 													
	\draw (2,-4.5) node {{\scalebox{0.8}{$m$}}} ; 								
\end{tikzpicture}}}~.
\]
We denote by $\nu=\alpha+\beta+\gamma$ the decomposition of $\nu$ according to these three types of composites. 
Recall that the operator $\rhd$ is linear on the right-hand side, that is $\Theta \rhd \nu= \Theta \rhd \alpha + \Theta \rhd \beta + 
\Theta \rhd \gamma$~. The first term $\Theta \rhd \alpha$ is made up composites of type A only. 
The terms in $\Theta \rhd \beta$ are of types B and C according to whether there exists an element of $\Theta$ connected (C) or not (B) to one output below and one output above the differential $\d$. The terms in $\Theta \rhd \gamma$ are made up of composites of type C only. 
Splitting apart the terms of type A (respectively B) whether the operation above (respectively below) the differential is trivial or not, we decompose further the elements $\Theta \rhd \nu$ into the following five types of composites: 
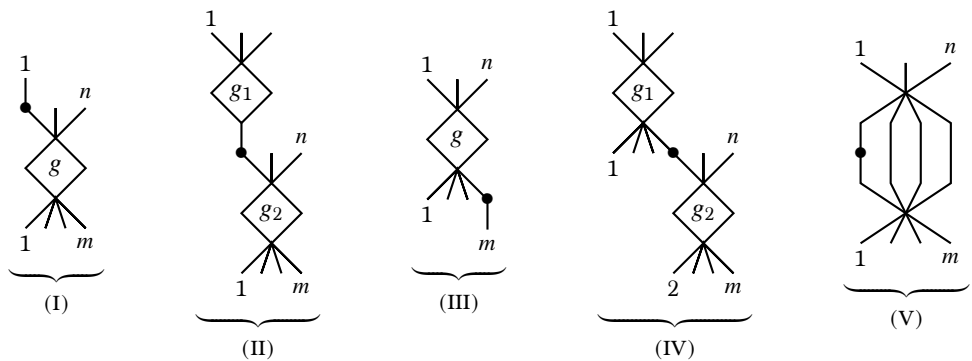
\begin{align*}
\underbrace{\vcenter{\hbox{\begin{tikzpicture}[scale=0.4]
	\draw[thick] (0,0)--(0,1);
	\draw[thick] (0,0)--(1,-1)--(1,0)--(1,-1)--(2,0);		
	\draw[thick] (1,-1)--(2,-2)--(1,-3)--(0,-2)--cycle;
	\draw[thick] (0,-4)--(1,-3)--(2,-4)--(1,-3)--(0.66,-4)--(1,-3)--(1.33,-4);
	\draw (0,0) node {{$\bullet$}} ; 	
	\draw (1,-2) node {\scalebox{0.9}{$g$}} ; 	
	\draw (0,1.5) node {{\scalebox{0.8}{$1$}}} ; 				
	\draw (2,0.5) node {{\scalebox{0.8}{$n$}}} ; 						
	\draw (0,-4.5) node {{\scalebox{0.8}{$1$}}} ; 							
	\draw (2,-4.5) node {{\scalebox{0.8}{$m$}}} ; 								
\end{tikzpicture}}}}_{\mathrm{(I)}}
\qquad \quad
\underbrace{\vcenter{\hbox{\begin{tikzpicture}[scale=0.4]
	\draw[thick] (0,0)--(0,1);
	\draw[thick] (0,1)--(-1, 2)--(0,3)--(1,2)--cycle;
	\draw[thick] (-1,4)--(0,3)--(0,4)--(0,3)--(1,4);
	\draw[thick] (0,0)--(1,-1)--(1,0)--(1,-1)--(2,0);		
	\draw[thick] (1,-1)--(2,-2)--(1,-3)--(0,-2)--cycle;
	\draw[thick] (0,-4)--(1,-3)--(2,-4)--(1,-3)--(0.66,-4)--(1,-3)--(1.33,-4);
	\draw (0,0) node {{$\bullet$}} ; 	
	\draw (0,2) node {\scalebox{0.9}{$g_1$}} ; 	
	\draw (1,-2) node {\scalebox{0.9}{$g_2$}} ; 	
	\draw (-1,4.5) node {{\scalebox{0.8}{$1$}}} ; 									
	\draw (2,0.5) node {{\scalebox{0.8}{$n$}}} ; 						
	\draw (0,-4.5) node {{\scalebox{0.8}{$1$}}} ; 							
	\draw (2,-4.5) node {{\scalebox{0.8}{$m$}}} ; 								
\end{tikzpicture}}}}_{\mathrm{(II)}}
\qquad \quad
\underbrace{\vcenter{\hbox{\begin{tikzpicture}[scale=0.4]
	\draw[thick] (0,0)--(0,-1);
	\draw[thick] (0,0)--(-1,1)--(-2,0)--(-1,1)--(-0.66,0)--(-1,1)--(-1.33,0);
	\draw[thick] (-1,1)--(-2, 2)--(-1,3)--(0,2)--cycle;
	\draw[thick] (-2,4)--(-1,3)--(-1,4)--(-1,3)--(0,4);
	\draw (0,0) node {{$\bullet$}} ; 	
	\draw (-1,2) node {\scalebox{0.9}{$g$}} ; 	
	\draw (-2,4.5) node {{\scalebox{0.8}{$1$}}} ; 				
	\draw (0,4.5) node {{\scalebox{0.8}{$n$}}} ; 					
	\draw (-2,-0.5) node {{\scalebox{0.8}{$1$}}} ; 							
	\draw (0,-1.5) node {{\scalebox{0.8}{$m$}}} ; 							
\end{tikzpicture}}}}_{\mathrm{(III)}}
\qquad \quad
\underbrace{\vcenter{\hbox{\begin{tikzpicture}[scale=0.4]
	\draw[thick] (0,0)--(-1,1)--(-2,0);
	\draw[thick] (0,0)--(-1,1)--(-2,0)--(-1,1)--(-0.66,0)--(-1,1)--(-1.33,0);
	\draw[thick] (-1,1)--(-2, 2)--(-1,3)--(0,2)--cycle;
	\draw[thick] (-2,4)--(-1,3)--(-1,4)--(-1,3)--(0,4);
	\draw[thick] (0,0)--(1,-1)--(1,0)--(1,-1)--(2,0);		
	\draw[thick] (1,-1)--(2,-2)--(1,-3)--(0,-2)--cycle;
	\draw[thick] (0,-4)--(1,-3)--(2,-4)--(1,-3)--(0.66,-4)--(1,-3)--(1.33,-4);
	\draw (0,0) node {{$\bullet$}} ; 	
	\draw (-1,2) node {\scalebox{0.9}{$g_1$}} ; 	
	\draw (1,-2) node {\scalebox{0.9}{$g_2$}} ; 	
	\draw (-2,4.5) node {{\scalebox{0.8}{$1$}}} ; 									
	\draw (2,0.5) node {{\scalebox{0.8}{$n$}}} ; 						
	\draw (-2,-0.5) node {{\scalebox{0.8}{$1$}}} ; 							
	\draw (0,-4.5) node {{\scalebox{0.8}{$2$}}} ; 							
	\draw (2,-4.5) node {{\scalebox{0.8}{$m$}}} ; 								
\end{tikzpicture}}}}_{\mathrm{(IV)}}
\qquad \quad
\underbrace{\vcenter{\hbox{\begin{tikzpicture}[scale=0.4]
	\draw[thick] (0,3)--(3,1)--(3,-1)--(0,-3);
	\draw[thick] (3,3)--(0,1)--(0,-1)--(3,-3);		
	\draw[thick] (1.5,3)--(1.5,2)--(1,1)--(1,-1)--(1.5,-2)--(1,-3);			
	\draw[thick] (1.5,3)--(1.5,2)--(2,1)--(2,-1)--(1.5,-2)--(2,-3);				
	\draw (0,0) node {{$\bullet$}} ; 	
	\draw (0,3.5) node {{\scalebox{0.8}{$1$}}} ; 				
	\draw (3,3.5) node {{\scalebox{0.8}{$n$}}} ; 					
	\draw (0,-3.5) node {{\scalebox{0.8}{$1$}}} ; 							
	\draw (3,-3.5) node {{\scalebox{0.8}{$m$}}} ; 								
\end{tikzpicture}}}}_{\mathrm{(V)}}
\end{align*}

In order to proceed further with the study of these five cases, let us introduce the positive integer $\mathrm{T}(m, i, g,p)$ which counts the number of ways to compose $p$ chosen indistinguishable outputs of a vertex of a graph with graphs of type $\Theta$ in order to produce a new graph with $i$ more distinguishable inputs, with  $m$ more distinguishable outputs, and with genus increased by $g$.
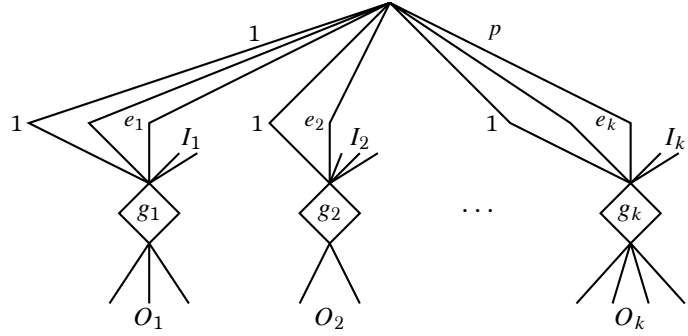
\[
\vcenter{\hbox{\begin{tikzpicture}[scale=0.8]
	\coordinate (A) at (0,-1);
	\coordinate (A0) at (-6,-3);	
	\coordinate (A1) at (-5,-3);	
	\coordinate (A2) at (-4,-3);		
	\coordinate (A3) at (-2,-3);		
	\coordinate (A3bis) at (-1,-3);					
	\coordinate (A4) at (2,-3);				
	\coordinate (A5) at (3,-3);					
	\coordinate (A6) at (4,-3);		
	\draw ($(A) + (1.5,-3.5)$) node {\scalebox{1}{$\cdots$}} ; 					
	\draw ($(A) + (-2.25,-0.5)$) node {\scalebox{0.8}{$1$}} ; 						
	\draw ($(A) + (1.75,-0.5)$) node {\scalebox{0.8}{$p$}} ; 						
	\coordinate (B) at (-4,-4.5);
	\draw[thick] ($(B) +(-0.5,0)$) --  ($(B)+(0,0.5)$) --  ($(B)+(0.5,0)$)--  ($(B)+(0,-0.5)$) -- cycle;
	\draw[thick] 
		($(B) +(0,0.5)$) -- (A0) -- (A)	
		($(B) +(0,0.5)$) -- (A1) -- (A)	
		($(B) +(0,0.5)$) -- (A2)-- (A);	
	\draw[thick] 
		($(B) +(0,-0.5)$) -- ($(B) +(0,-0.5)+(-0.66,-1)$)
		($(B) +(0,-0.5)$) -- ($(B) +(0,-0.5)+(0,-1)$)
		($(B) +(0,-0.5)$) -- ($(B) +(0,-0.5)+(0.66,-1)$);	
	\draw (B) node {\scalebox{0.9}{$g_1$}} ; 						
	\draw ($(B) + (0,-1.75)$) node {\scalebox{0.9}{$O_1$}} ; 
	\draw[thick]
		($(B) +(0,0.5)$) -- ++(0.5, 0.5)
		($(B) +(0,0.5)$) -- ++(0.8, 0.5);		
	\draw ($(B) + (0,0.5) + (0.7, 0.75)$) node {\scalebox{0.9}{$I_1$}} ; 		
	\draw ($(A0)+(-0.2,0)$) node {\scalebox{0.8}{$1$}} ; 			
	\draw ($(A2)+(-0.25,0)$) node {\scalebox{0.8}{$e_1$}} ; 				
	\coordinate (B) at (-1,-4.5);
	\draw[thick] ($(B) +(-0.5,0)$) --  ($(B)+(0,0.5)$) --  ($(B)+(0.5,0)$)--  ($(B)+(0,-0.5)$) -- cycle;
	\draw[thick] 
		($(B) +(0,0.5)$) -- (A3)-- (A)
		($(B) +(0,0.5)$) -- (A3bis)-- (A);			
	\draw[thick] 
		($(B) +(0,-0.5)$) -- ($(B) +(0,-0.5)+(-0.5,-1)$)
		($(B) +(0,-0.5)$) -- ($(B) +(0,-0.5)+(0.5,-1)$);	
	\draw (B) node {\scalebox{0.9}{$g_2$}} ; 
	\draw ($(B) + (0,-1.75)$) node {\scalebox{0.9}{$O_2$}} ; 														
	\draw[thick]
		($(B) +(0,0.5)$) -- ++(0.2, 0.5)
		($(B) +(0,0.5)$) -- ++(0.5, 0.5)
		($(B) +(0,0.5)$) -- ++(0.8, 0.5);		
	\draw ($(B) + (0,0.5) + (0.5, 0.75)$) node {\scalebox{0.9}{$I_2$}};
	\draw ($(A3)+(-0.2,0)$) node {\scalebox{0.8}{$1$}} ; 			
	\draw ($(A3bis)+(-0.25,0)$) node {\scalebox{0.8}{$e_2$}} ; 					 			
	\coordinate (B) at (4,-4.5);
	\draw[thick] ($(B) +(-0.5,0)$) --  ($(B)+(0,0.5)$) --  ($(B)+(0.5,0)$)--  ($(B)+(0,-0.5)$) -- cycle;
	\draw[thick] 
		($(B) +(0,0.5)$) -- (A4)-- (A)	
		($(B) +(0,0.5)$) -- (A5)-- (A)				
		($(B) +(0,0.5)$) -- (A6)-- (A);
	\draw[thick] 
		($(B) +(0,-0.5)$) -- ($(B) +(0,-0.5)+(-0.9,-1)$)
		($(B) +(0,-0.5)$) -- ($(B) +(0,-0.5)+(-0.3,-1)$)
		($(B) +(0,-0.5)$) -- ($(B) +(0,-0.5)+(0.3,-1)$)		
		($(B) +(0,-0.5)$) -- ($(B) +(0,-0.5)+(0.9,-1)$);	
	\draw (B) node {\scalebox{0.9}{$g_k$}}; 											
	\draw ($(B) + (0,-1.75)$) node {\scalebox{0.9}{$O_k$}}; 
	\draw[thick]
		($(B) +(0,0.5)$) -- ++(0.5, 0.5)
		($(B) +(0,0.5)$) -- ++(0.8, 0.5);			
	\draw ($(B) + (0,0.5) + (0.75, 0.75)$) node {\scalebox{0.9}{$I_k$}};
	\draw ($(A4)+(-0.3,0)$) node {\scalebox{0.8}{$1$}} ; 			
	\draw ($(A6)+(-0.4,0)$) node {\scalebox{0.8}{$e_k$}} ; 					
\end{tikzpicture}}}
\]
Let $k\geqslant 1$ denote the number of graphs of type $\Theta$ on the bottom level. 
Grafting the $p$ ouputs to these $k$ graphs already increases the genus by $p-k$. 
Since the total increase of genus should be equal to $g$, one must have $k\geqslant p-g$. 
For any $\max(1, p-g) \leqslant k\leqslant p$, we consider first a partition $O_1 \sqcup O_2 \sqcup \cdots \sqcup O_k=[m]$ of the set 
$[m]=\{1, \ldots ,m\}$ into $k$ non-empty blocks. 
Their number is given by the Stirling number of second kind $\mathrm{S}(m,k)$~. 
We choose to order the blocks of the partition increasingly according to their smallest elements. 
Then, we consider an ordered partition 
$I_1 \sqcup I_2 \sqcup \cdots \sqcup I_k=[i]$ of the set 
$[i]=\{1, \ldots ,i\}$ into $k$ possibly empty blocks. 
There are $k^i$ possible choices. 
After that, we consider a composition $e_1+e_2+\cdots+e_k=p$ of the integer $p$ into $k$ (positive) parts: there are $\binom{p-1}{k-1}$ of them. 
Finally, it remains to match the require genus: for that, we consider a weak composition 
$g_1+g_2+\cdots+g_k=g-p+k$
of the integer $g-p+k$ into $k$ (non-negative) parts. Their number is equal to $\binom{g-p+2k-1}{k-1}$. In the end, we get the value of the combinatorial coefficient: 
\[
\mathrm{T}(m, i, g,p)=\sum_{k=\max(1, p-g)}^p \mathrm{S}(m,k)k^i \binom{p-1}{k-1}\binom{g-p+2k-1}{k-1}~. 
\]

For any $m,n\geqslant 1$,  $g\geqslant 0$, such that $m+n+g\geqslant 3$, the coefficient of the composite of type  (I) 
on the left-hand side of Equation~\eqref{eq:CRUCIAL} is equal to $-1$ and the coefficient of the right-hand side is equal to 
\begin{equation*}
\sum_{\substack{
0 \leqslant l \leqslant n-1\\
0\leqslant g_2 \leqslant g\\
1\leqslant p\leqslant m+g-g_2\\
(p,l,g_2)\neq (1,0,0)}}
\binom{n-1}{l} A^{g_2}_{p,l}\, \mathrm{T}(m, n-l-1, g-g_2,p)~.
\end{equation*}
This is the coefficient of the elements in $\Theta \rhd \alpha$ of type (I) obtained by composing all the outputs of one element of type A 
(with one input labelled by the differential, $l$ other inputs, $p$ outputs, and genus $g_2$)
above with some inputs of a level of elements of $\Theta$ below. The missing index $(1,0,0)$ would correspond to the composite 
where the element of type A above is the sole differential: such a term cannot be considered since the differential alone is not part of $\nu$. 

\medskip

We define the \emph{weight} of a triple $(m,n,g)$ by $m+n+2g-2$. It  corresponds to the number of generators making up the element of the properad $\mathrm{Frob}$ having arity $(m,n)$ and genus $g$. 
For the triple $(1,1,0)$ of weight 0, Equation~\eqref{eq:CRUCIAL} is trivially true. 
For the two triples $(1,2,0)$ and $(2,1,0)$ of weight 1, the coefficients of the components of Equation~\eqref{eq:CRUCIAL} of type (I) give respectively $-1 = A^0_{1,1}$ and $-1=A^0_{2,0}$~. For any triple $(m,n,g)$ of 
higher weight, the right-hand side of the equation is made up of a sum where the first term is equal to $A^{g}_{m, n-1}$ and where all the other terms have strictly lower weight. Working by induction on the weight,  this prescribes the values of all the coefficients $A^{g}_{m, n-1}$, except $A^0_{1,0}$~. 

\medskip

For any $m,n\geqslant 1$,  $g\geqslant 0$, such that $m+n+g\geqslant 3$, 
when the composite of type (II) presents a trivial operation below the differential, 
the coefficient on the left-hand side of Equation~\eqref{eq:CRUCIAL} is equal to $1$ and the coefficient of the right-hand side is equal to 
$A^0_{1,0}$~, which forces it to be equal to $1$. 
When the composite of type (II) presents a non-trivial operation below the differential, 
the coefficient 
on the left-hand side of Equation~\eqref{eq:CRUCIAL} is equal to $0$ and the coefficient of the right-hand side is equal to 
\begin{equation*}
\sum_{\substack{
0 \leqslant l \leqslant n-1\\
0\leqslant g_2 \leqslant g\\
1\leqslant p\leqslant m+g-g_2}}
\binom{n-1}{l} A^{g_2}_{p,l}\, \mathrm{T}(m, n-l-1, g-g_2,p)~. 
\end{equation*}
Since there is now a non-trivial operation above the differential, all the composite in $\Theta \rhd \alpha$ are present including the case 
$(p,l, g_2)=(1,0,0)$ corresponding to the coefficient of $A^0_{1,0}=1$~. Therefore, all these equations are satisfied by the preceding case. 

\medskip

For any $m\geqslant 2$, $n\geqslant 1$, and $g\geqslant 0$, the coefficient of the composite of type (III) 
on the left-hand side of Equation~\eqref{eq:CRUCIAL} is equal to $1$ and the coefficient of the right-hand side is equal to 
\begin{equation*}
\sum_{l=0}^{n-1}
\binom{n}{l}  (-1)^{n-l-1}A^0_{1,0}~.
\end{equation*}
Indeed, the element in $\Theta \rhd \beta$ of type (III) of arity $(m,n)$ and genus $g$ is obtained as a sum of compositions 
of a graph of type B, with trivial operation below the differential, and elements of type $\Theta$ below. 
In $\beta$, the only graphs of type B, with trivial operation below the differential, having non-trivial coefficient are the ones with 2 outputs and genus 0. 
So the only non-trivial composites in $\Theta \rhd \beta$ of type (III) are obtained when one composes such a graph above with only one graph of type $\Theta$, which must be of arity $(m-1, l+1)$, with $0\leqslant n-1$, and genus $g$. In this case, the coefficient is equal 
$\binom{n}{l} \, (-1)^{n-l-1}A^0_{1,0}$~. Since $A^0_{1,0}=1$, all these equations hold. 

\medskip

Regarding the composites of type (IV), we denote respectively by $m'$ and $n'$ the number of outputs and inputs of the sub-graph above the differential and we denote respectively by $m''$ and $n''$ the number of outputs and inputs of the sub-graph below the differential. 
Since the only graphs of type B in $\beta$ have 2 outputs and genus 0 above the differential, the only way to get composites of type (IV) by grafting below graphs of type $\Theta$ amounts to composing the unique left output above the differential to one graph of type $\Theta$ having genus $g_1$ and $m'$ outputs. So in this case, the coefficient of the left-hand side of Equation~\eqref{eq:CRUCIAL} is equal to $0$ and the coefficient of the right-hand side is equal to 
\begin{align*}
&\sum_{\substack{1\leqslant k \leqslant n'\\
0 \leqslant l \leqslant n''\\
0\leqslant g'_2 \leqslant g_2\\
1\leqslant p\leqslant m''+g_2-g'_2}}
\binom{n'}{k}(-1)^{k-1}
\binom{n''}{l} A^{g'_2}_{p,l}\, \mathrm{T}(m'', n''-l, g_2-g'_2,p)
=\\
&\sum_{\substack{1\leqslant k \leqslant n'}}
\binom{n'}{k}(-1)^{k-1}
\sum_{\substack{
0 \leqslant l \leqslant n''\\
0\leqslant g'_2 \leqslant g_2\\
1\leqslant p\leqslant m''+g_2-g'_2}}
\binom{n''}{l}  A^{g'_2}_{p,l}\, \mathrm{T}(m'', n''-l, g_2-g'_2,p)~. 
\end{align*}
Since $m''+n''+g_2\geqslant 2$, the last sums is equal to $0$ by case (II). 

\medskip

Finally, it remains to treat the case of composites of type (V), for any $m,n,g \geqslant 1$~. 
As explained above, their can appear from either 
$\Theta \rhd \beta$ or $\Theta \rhd \gamma$ in $\Theta \rhd \nu$. 
The first way to get composites of type (V) amounts to composing one graph of type B above a level of graphs of type $\Theta$, where one such graph is connected to the input above the differential and to at least one output below the differential. 
Let us denote respectively by $n'$ and $l$ the number of inputs above and below the differential of the graph of type B, which has genus equal to $g_2$ and a total of $p+1$ outputs. 
The coefficient coming from such a composite is equal to 
\[
(-1)^{n'-1}A^{g_2}_{p,l} \mathrm{T}(m, n-n'-l+1, g-g_2-1, p)~, 
\]
since it is given by the number of ways of composing such a graph of type B above with $k$ graphs of type $\Theta$ below with one more distinguished input than $n-n'-l$, which is connected to the input coming  from above the differential. 
Considering all such composites, we get the following overall coefficient 
\begin{align}\tag{$\looparrowleft$}\label{eq:TermeB}
\sum_{\substack{
1\leqslant n'\leqslant n\\ 
0 \leqslant l \leqslant n-n'\\
0\leqslant g_2 \leqslant g-1\\
1\leqslant p\leqslant m+g-g_2-1}}
\binom{n}{n', l, n-n'-l}
(-1)^{n'-1}A^{g_2}_{p,l} \mathrm{T}(m, n-n'-l+1, g-g_2-1, p)~, 
\end{align}
The second way to get composites of type (V) amounts to composing one graph of type C above a level of graphs of type $\Theta$, which produces the coefficient 
\begin{equation}\tag{$\looparrowright$}\label{eq:TermeC}
\sum_{\substack{
1 \leqslant l \leqslant n\\
1\leqslant g_2 \leqslant g\\
1\leqslant p\leqslant m+g-g_2}}
\binom{n}{l} C^{g_2}_{p,l}\, \mathrm{T}(m, n-l, g-g_2,p)~. 
\end{equation}
For any $m,n,g \geqslant 1$, the coefficient of the composite of type  (V) 
on the left-hand side of Equation~\eqref{eq:CRUCIAL} is equal to $0$ and the coefficient of the right-hand side is equal to the sum of the two above terms \eqref{eq:TermeB} and \eqref{eq:TermeC}. 
The case $m=n=g=1$ gives $C^1_{1,1}=-A^{0}_{1,0}=-1$. 
Then, for triples $(m,n,g)$ of higher weight, the term \eqref{eq:TermeB} is known and the term \eqref{eq:TermeC} is made up of $C^g_{m,n}$ plus terms of stricly lower weight: this prescribes the value of $C^g_{m,n}$ by induction. 

\medskip

Let us now show that the recursive formula 
\begin{equation*}
\sum_{\substack{
0 \leqslant l \leqslant n\\
0\leqslant g' \leqslant g\\
1\leqslant p\leqslant m+g-g'}}
\binom{n}{l} A^{g'}_{p,l}\, \mathrm{T}(m, n-l, g-g',p)=0~, 
\end{equation*}
for $m+n+2g\geqslant 2$, with initial value $A^0_{1,0}=1$,
admits for solutions
\begin{align*}
A^g_{m,n}=\sum_{k=1}^{m+n+2g-1} (-1)^k
\sum_{\Xi_k(m,n,g)} \prod_{j=2}^k \mathrm{T}(p_j, |I_j|, g_j, p_{j-1})~, 
 \end{align*}
where the sums run over the sets $\Xi_k(m,n,g)$ made up of the partitions $I_1 \sqcup I_2 \sqcup \cdots \sqcup I_k=\{2, \ldots, n\}$, the integer partitions $g_1+g_2+\cdots+g_k=g$, and the integers $p_1, \ldots, p_{k-1}$ such that 
$1\leqslant p_1 \leqslant p_2+g_2 \leqslant p_3 + g_2+g_3 \leqslant p_{k-1}+g_2+\cdots+g_{k-1}\leqslant 
m+g_2+\cdots+g_k$, satisfying $p_{j}-p_{j-1}+|I_j|+2g_j\geqslant 1$, for any $1\leqslant j\leqslant k$, under the convention $p_0=1$ and $p_k=m$. It is understood that for $k=1$ the product is equal to $1$. 
Heuristically speaking, the recursive formula is the fixed point equation $\Theta \rhd (\alpha+\delta)=0$, whose solution is given by 
the series $\alpha=\sum_{k\geqslant 1} (-1)^k \underbrace{\Theta \, \bar\rhd\, (\Theta \,\bar\rhd\, ( \cdots (\Theta \,\bar\rhd\,(\Theta \,\bar\rhd\,}_{k \ \text{times}} \delta))))$~, where $\bar\rhd$ means that we do not allow only identities from $\Theta$.

\medskip

We prove this claim by induction on the weight $m+n+2g-1$. For $m+n+2g-1=1$, it is straightforward to compute 
$A^0_{1,1}=-1$ and $A^0_{2,0}=-1$ from the recursive formula. This coincides with the part $k=1$ of the above given solution. Suppose now that the result holds true up to weight $m+n+2g-2$. Since $\mathrm{T}(m,0,0,m)=1$, we get the coefficient $A^g_{m,n}$ as follows 
\begin{align*}
A^g_{m,n}&= - \sum_{\substack{
0 \leqslant l \leqslant n\\
0\leqslant g' \leqslant g\\
1\leqslant p\leqslant m+g-g'\\
(p,l,g')\neq(m,n,g)\\
(p,l,g')\neq(1,0,0)}}
\binom{n}{l} A^{g'}_{p,l}\, \mathrm{T}(m, n-l, g-g',p)
-A^0_{1,0}
\\
&=- \sum_{\substack{
0 \leqslant l \leqslant n\\
0\leqslant g' \leqslant g\\
1\leqslant p\leqslant m+g-g'\\
(p,l,g')\neq(m,n,g)\\
(p,l,g')\neq(1,0,0)}}
\binom{n}{l} 
\left(
\sum_{k=1}^{p+l+2g'-1} (-1)^k
\sum_{\Xi_k(p,l,g')} \prod_{j=2}^k \mathrm{T}(p_j, |I_j|, g_j, p_{j-1})
\right)
\, \mathrm{T}(m, n-l, g-g',p)-1\\
&=
\sum_{k+1=2}^{m+n+2g-1} (-1)^{k+1}
\sum_{\Xi_{k+1}(m,n,g)} \prod_{j=2}^{k+1} \mathrm{T}(p_j, |I_j|, g_j, p_{j-1})-1\\
&=
\sum_{k=1}^{m+n+2g-1} (-1)^k
\sum_{\Xi_k(m,n,g)} \prod_{j=2}^k \mathrm{T}(p_j, |I_j|, g_j, p_{j-1})~, 
\end{align*}
where the new level of indices is given by $g_{k+1}=g'$, $p_k=p$, and where the subset $I_{k+1}$ is made up of the $l$ elements from $\{2, \ldots, n\}$ chosen by the binomial coefficient.

\medskip

In order to establish the last formula for the coefficients $C$, we denote by $Q(m,n,g)$ the formula given in \eqref{eq:TermeB}, for $m,n,g\geqslant 0$, and let us prove that it is equal to 
$Q(m,n,1)=(-1)^{m+n}$, for $g=1$, and to $Q(m,n,g)=0$, for $g>1$. First, we notice that 
\begin{align*}
Q(m,n,g)&= 
\sum_{\substack{
1\leqslant n'\leqslant n\\ 
0 \leqslant l \leqslant n-n'\\
0\leqslant g_2 \leqslant g-1\\
1\leqslant p\leqslant m+g-g_2-1}}
\binom{n}{n', l, n-n'-l}
(-1)^{n'-1}A^{g_2}_{p,l} \mathrm{T}(m, n-n'-l+1, g-g_2-1, p)\\
&=
\sum_{1\leqslant n'\leqslant n} 
(-1)^{n'-1} \binom{n}{n'}
\underbrace{
\sum_{\substack{
0 \leqslant l \leqslant n-n'\\
0\leqslant g_2 \leqslant g-1\\
1\leqslant p\leqslant m+g-g_2-1}}
\binom{n-n'}{l}
A^{g_2}_{p,l} \mathrm{T}(m, n-n'-l+1, g-g_2-1, p)
}_{D(m,n-n',g-1)\coloneq}
\end{align*}
We claim that $D(m,0,0)=(-1)^{m-1}$, for $m\geqslant 1$, and $D(m,n,g)=0$, when $(n,g)\neq (0,0)$. This would prove the abovementioned formula since then
\begin{align*}
Q(m,n,1)=\sum_{1\leqslant n'\leqslant n} 
(-1)^{n'-1} \binom{n}{n'}
D(m,n-n',0)=
(-1)^{n-1}(-1)^{m-1}=(-1)^{m+n}
\end{align*}
and 
\begin{align*}
Q(m,n,g)=
\sum_{1\leqslant n'\leqslant n} 
(-1)^{n'-1} \binom{n}{n'}
D(m,n-n',g-1)=0~, 
\end{align*}
for $g\geqslant 1$. 
The first formula for $D$ is straightforward to establish: 
\begin{align*}
D(m,0,0)&=
\sum_{1\leqslant p\leqslant m}
A^{0}_{p,0} \mathrm{T}(m, 1, 0, p)
=
\sum_{1\leqslant p\leqslant m}
(-1)^{p-1}(p-1)!pS(m,p)
\\
&=-\sum_{1\leqslant p\leqslant m}
(-1)^{p}p!S(m,p)
=(-1)^{m-1}~. 
\end{align*}
Regarding the second case $(m,n,g)$ for $(n,g)\neq(0,0)$, we are going to use the aforementioned formula for the coefficients $A$, which shows that $A(m,n,g)$ is equal to the signed sum with coefficient $(-1)^k$ of the number of ways of decomposing the operation of genus $g$ with $n$ inputs and $m$ outputs into $k$ non-trivial levels with one operation at the top carrying the first input, with all the outputs at the bottom, and such that, for each level, the upward directed graph is connected. Including this formula inside the definition of $D$, we get the same interpretation for $D(m,n,g)$, except that the bottom level of each directed connected graph carries the input $n+1$. 
\begin{figure*}[h]
\begin{align*}
(-1)^3\ \vcenter{\hbox{\begin{tikzpicture}[scale=0.5]
	\draw[thick] (0,2)--(0,0)--(1,-1);
	\draw[thick] (0,0)--(1,-1)--(1.5,-0.5);		
	\draw[thick] (1,-1)--(1.5,-1.5)--(1,-2)--(0.5,-1.5)--cycle;
	\draw[thick] (1,-2)--(0,-3)--(1,-2)--(2,-3);
	\draw[thick] (1,-1)--(1,-2);			
	\draw[thick] (1,-2)--(0,-3)--(0,-4)--(0.5, -4.5)--(0,-5)--(-0.5, -4.5)--(0, -4);	
	\draw[thick] (0,-5)--(-1, -6)--(0,-5)--(0,-6)--(0,-5)--(1,-6);
	\draw[thick] (1,-2)--(2,-3)--(2,-6);
	\draw[thick] (2,-4)--(2.5,-3.5);	
	\draw[thick] (0,-5)--(2,-7);	
	\draw[thick] (2,-5)--(2,-7)--(2.5, -7.5)--(2, -8)--(1.5, -7.5)--(2, -7);		
	\draw[thick] (2.5, -7.5)--(1, -9)--(2, -8)--(2, -9)--(2,-8)--(3,-9);			
	\draw[thick] (0,-5)--(-1,-6)--(-1, -7)--(-0.5, -7.5)--(-1,-8)--(-1.5, -7.5)--(-1, -7);	
	\draw[thick] (-1,-8)--(-2,-9)--(-1,-8)--(0,-9);		
	\draw[thick] (0,-5)--(0,-6)--(-1, -7);
	\draw[thick] (-1,-7)--(-1,-8);	
	\draw[thick] (2,-7)--(2.5,-6.5);							
	\draw[thick] (-1,-7)--(-1.5,-6.5);								
	\draw[thin, dashed] (-2,0)--(3,0);
	\draw[thin, dashed] (-2,-3)--(3,-3);	
	\draw[thin, dashed] (-2,-6)--(3,-6);		
	\draw (0,1) node {{$\bullet$}} ; 	
	\draw (0,2.5) node {{\scalebox{0.8}{$1$}}} ; 				
	\draw (2,-0.5) node {{\scalebox{0.8}{$4$}}} ; 						
	\draw (3,-3.5) node {{\scalebox{0.8}{$3$}}} ; 		
	\draw (-2,-6.5) node {{\scalebox{0.8}{$2$}}} ; 	
	\draw (3,-6.5) node {{\scalebox{0.8}{$\textbf{5}$}}} ; 		
	\draw (0,-9.5) node {{\scalebox{0.8}{$5$}}} ; 							
	\draw (-2,-9.5) node {{\scalebox{0.8}{$2$}}} ; 							
	\draw (1,-9.5) node {{\scalebox{0.8}{$1$}}} ; 							
	\draw (2,-9.5) node {{\scalebox{0.8}{$3$}}} ; 							
	\draw (3,-9.5) node {{\scalebox{0.8}{$4$}}} ; 	
\end{tikzpicture}}}
\end{align*}
\caption{A $3$-leveled graph appearing in the computation of $D(5,4,6)$}
		\label{Fig:Pyr}
	\end{figure*}
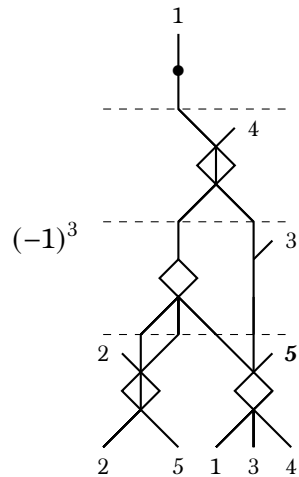

When $(n,g)\neq(0,0)$, each of such leveled directed graph is a composition which admits at least one operation $\BB$ with two inputs and one output. 
When one reads the graph from top to bottom, there are three possible situations for the first product to appear. 
\begin{enumerate}
\item One can encounter only levels made up of coproducts $\CC$ before to reach a level with a product having one input labeled by an element of $\{2, \ldots, n\}$; we consider the smallest such element $i$. Then there are two cases: either this product with input labeled by $i$ is alone on that level or not. 
In this last case, there is a unique way to pull up the product with input labeled by $i$ on an new upper level. 
These two cases produce graphs which are in one-to-one correspondence but which appear with different signs; so they cancel. 
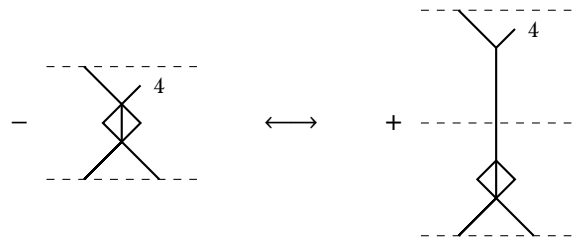
\begin{align*}
- \ \ \vcenter{\hbox{\begin{tikzpicture}[scale=0.5]
	\draw[thick] (0,0)--(1,-1)--(1.5,-0.5);		
	\draw[thick] (1,-1)--(1.5,-1.5)--(1,-2)--(0.5,-1.5)--cycle;
	\draw[thick] (1,-2)--(0,-3)--(1,-2)--(2,-3);
	\draw[thick] (1,-1)--(1,-2);			
	\draw[thin, dashed] (-1,0)--(3,0);
	\draw[thin, dashed] (-1,-3)--(3,-3);	
	\draw (2,-0.5) node {{\scalebox{0.8}{$4$}}} ; 							
\end{tikzpicture}}}
\qquad 
\longleftrightarrow
\qquad 
+ \ \ 
\vcenter{\hbox{\begin{tikzpicture}[scale=0.5]
	\draw[thick] (0,0)--(1,-1)--(1.5,-0.5);		
	\draw[thick] (1,-4)--(1.5,-4.5)--(1,-5)--(0.5,-4.5)--cycle;
	\draw[thick] (1,-5)--(0,-6)--(1,-5)--(2,-6);
	\draw[thick] (1,-1)--(1,-5);			
	\draw[thin, dashed] (-1,0)--(3,0);
	\draw[thin, dashed] (-1,-3)--(3,-3);	
	\draw[thin, dashed] (-1,-6)--(3,-6);		
	\draw (2,-0.5) node {{\scalebox{0.8}{$4$}}} ; 							
\end{tikzpicture}}}
\end{align*}
\item If the first situation does not happen, then it is possible that the first product(s) encountered from top to bottom are at the top of operations 
$\vcenter{\hbox{\begin{tikzpicture}[baseline=1.8ex,scale=0.1]
	\draw[thick] (2,4) -- (2,2)--(4,4)--(2,2)--(0,4);
	\draw[thick] (4,0) -- (2,2);
	\draw[thick] (2,2) -- (0,0) -- (2,-2);
	\draw[thick] (2,-2) -- (4,0) -- (2,2);
	\draw[thick] (4,-4)--(2,-2) --(0,-4);	
\end{tikzpicture}}}$, that is with inputs grafted to upper edges and not part of a ``diamond'' creating an internal genus. 
Then there are two cases: either the level is made up of products or not. 
In this last case, we pull up all the products above the diamonds and the coproducts to create a new level. 
These two cases produce graphs which are in one-to-one correspondence but which appear with different signs; so they cancel. 
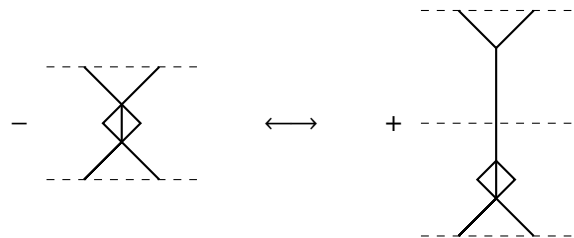
\begin{align*}
- \ \ \vcenter{\hbox{\begin{tikzpicture}[scale=0.5]
	\draw[thick] (0,0)--(1,-1)--(2,0);		
	\draw[thick] (1,-1)--(1.5,-1.5)--(1,-2)--(0.5,-1.5)--cycle;
	\draw[thick] (1,-2)--(0,-3)--(1,-2)--(2,-3);
	\draw[thick] (1,-1)--(1,-2);			
	\draw[thin, dashed] (-1,0)--(3,0);
	\draw[thin, dashed] (-1,-3)--(3,-3);							
\end{tikzpicture}}}
\qquad 
\longleftrightarrow
\qquad 
+ \ \ 
\vcenter{\hbox{\begin{tikzpicture}[scale=0.5]
	\draw[thick] (0,0)--(1,-1)--(2,0);		
	\draw[thick] (1,-4)--(1.5,-4.5)--(1,-5)--(0.5,-4.5)--cycle;
	\draw[thick] (1,-5)--(0,-6)--(1,-5)--(2,-6);
	\draw[thick] (1,-1)--(1,-5);			
	\draw[thin, dashed] (-1,0)--(3,0);
	\draw[thin, dashed] (-1,-3)--(3,-3);	
	\draw[thin, dashed] (-1,-6)--(3,-6);		
\end{tikzpicture}}}
\end{align*}
\item Finally, if it is neither of these two situations, it means that the first product encountered from top to bottom is part of a diamond. 
Then there are two cases: either the level is made up of ``pure'' diamonds, that is having no coproduct below, or not. 
In this last case, we pull up all the diamonds creating a new level, leaving the coproducts on the level below. 
These two cases produce graphs which are in one-to-one correspondence but which appear with different signs; so they cancel. 
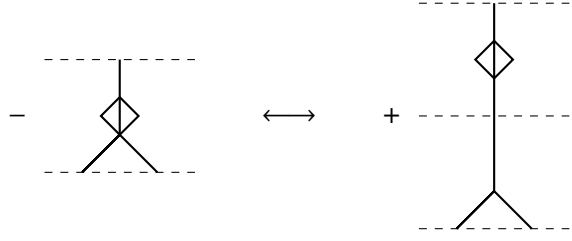
\begin{align*}
- \ \ \vcenter{\hbox{\begin{tikzpicture}[scale=0.5]
	\draw[thick] (1,0)--(1,-1);		
	\draw[thick] (1,-1)--(1.5,-1.5)--(1,-2)--(0.5,-1.5)--cycle;
	\draw[thick] (1,-2)--(0,-3)--(1,-2)--(2,-3);
	\draw[thick] (1,-1)--(1,-2);			
	\draw[thin, dashed] (-1,0)--(3,0);
	\draw[thin, dashed] (-1,-3)--(3,-3);							
\end{tikzpicture}}}
\qquad 
\longleftrightarrow
\qquad 
+ \ \ 
\vcenter{\hbox{\begin{tikzpicture}[scale=0.5]
	\draw[thick] (1,0)--(1,-1);		
	\draw[thick] (1,-1)--(1.5,-1.5)--(1,-2)--(0.5,-1.5)--cycle;
	\draw[thick] (1,-5)--(0,-6)--(1,-5)--(2,-6);
	\draw[thick] (1,-1)--(1,-5);			
	\draw[thin, dashed] (-1,0)--(3,0);
	\draw[thin, dashed] (-1,-3)--(3,-3);	
	\draw[thin, dashed] (-1,-6)--(3,-6);		
\end{tikzpicture}}}
\end{align*}
\end{enumerate}
In the end, this proves that $D(m,n,g)=0$, for $(n,g)\neq(0,0)$.

\medskip

With the above formula for the coefficients $Q(m,n,g)$, the recursive formula defining the coefficients $C^g_{m,n}$ takes the following simpler form
\begin{equation*}
\sum_{\substack{
1 \leqslant l \leqslant n\\
1\leqslant g' \leqslant g\\
1\leqslant p\leqslant m+g-g_2}}
\binom{n}{l} C^{g'}_{p,l}\, \mathrm{T}(m, n-l, g-g',p)=
\left\{
\begin{array}{ll}
(-1)^{m+n+1}~, & \text{for}\ g=1~, \\
0~,  & \text{for}\ g\geqslant 2~, 
\end{array}
\right.
\end{equation*}
for $m,n,g\geqslant 1$, with initial value $C^1_{1,1}=-1$. 
It admits the following solutions 
\begin{align*}
C^g_{m,n}= 
\sum_{k=1}^{m+n+2g-4} (-1)^{k+1}
\sum_{\Xi_k(m,n,g)} \left(n+(-1)^{p_1+|I_1|}\right)\prod_{j=2}^k \mathrm{T}(p_j, |I_j|, g_j, p_{j-1})~,
\end{align*}
for $m+n+2g\geqslant 2$.
The proof by induction is similar to the one given above for the coefficients $A$, which  concludes the proof of \cref{thm:KoszulHierarchy}.
\end{proof}

\begin{corollary}
Let $(A, \d)$ be a chain complex whose underling graded module is equipped with a Frobenius bialgebra structure $(A, \mu, \Delta)$. The formula 
\begin{align*}\label{eq:hierarchy}
\nu_{m,n}=&
\sum_{\substack{
k+l=n \\ 
\sigma\in \mathsf{Sh}(k, l)}} 
(-1)^{m+l-1} m! m^{l-1}\, 
\left(\left(\mu_{m, l+1}^{0} \circ_1^1 \d\right) \circ_1^1 \mu_{1, k}^{0}\right) \cdot \sigma\\
&+ 
\sum_{\substack{
k+l=n \\ 
\sigma\in \mathsf{Sh}(k, l)\\
\tau \in \mathsf{Sh}(1, m-1) }} 
(-1)^{m+n-1}(m-1)!(m-1)^{l-1}
\, 
\tau^{-1}\cdot \left(\left(\mu_{m-1, l+1}^{0} \circ_1^1 \d\right) \circ_1^2 \mu_{2, k}^{0}\right) \cdot \sigma~, 
\end{align*}
for $m, n\geqslant 1$ and $m+n\geqslant 3$, defines a gauge homotopy trivial shifted homotopy Lie bialgebra. 
\end{corollary}

\begin{figure*}[h]
\begin{align*}
\nu_{m,n}=&
\sum_{\substack{
k+l=n \\ 
\sigma\in \mathsf{Sh}(k, l)}} 
(-1)^{m+l-1} m! m^{l-1}
\vcenter{\hbox{\begin{tikzpicture}[scale=0.4]
	\draw[thick] (0,0)--(0,1);
	\draw[thick] (-1,2)--(0,1)--(0,2)--(0,1)--(1,2);
	\draw[thick] (0,0)--(1,-1)--(1,0)--(1,-1)--(2,0);		
	\draw[thick] (0,-2)--(1,-1)--(2,-2)--(1,-1)--(0.66,-2)--(1,-1)--(1.33,-2);
	\draw (0,0) node {{$\bullet$}} ; 	
	\draw (-1.25,2.5) node {{\scalebox{0.8}{$\sigma^{-1}(1)$}}} ; 				
	\draw (2.5,0.5) node {{\scalebox{0.8}{$\sigma^{-1}(n)$}}} ; 						
	\draw (0,-2.5) node {{\scalebox{0.8}{$1$}}} ; 							
	\draw (2,-2.5) node {{\scalebox{0.8}{$m$}}} ; 								
\end{tikzpicture}}}\\&
+\sum_{\substack{
k+l=n \\ 
\sigma\in \mathsf{Sh}(k, l)\\
\tau \in \mathsf{Sh}(1, m-1) }} 
(-1)^{m+n-1}(m-1)!(m-1)^{l-1}
\vcenter{\hbox{\begin{tikzpicture}[scale=0.4]
	\draw[thick] (0,0)--(-1,1)--(-2,0);
	\draw[thick] (-2,2)--(-1,1)--(-1,2)--(-1,1)--(0,2);
	\draw[thick] (0,0)--(1,-1)--(1,0)--(1,-1)--(2,0);		
	\draw[thick] (0,-2)--(1,-1)--(2,-2)--(1,-1)--(0.66,-2)--(1,-1)--(1.33,-2);
	\draw (0,0) node {{$\bullet$}} ; 		
	\draw (-2.25,2.5) node {{\scalebox{0.8}{$\sigma^{-1}(1)$}}} ; 				
	\draw (2.5,0.5) node {{\scalebox{0.8}{$\sigma^{-1}(n)$}}} ; 						
	\draw (-2.25,-0.5) node {{\scalebox{0.8}{$\tau(1)$}}} ; 							
	\draw (-0.25,-2.5) node {{\scalebox{0.8}{$\tau(2)$}}} ; 							
	\draw (2.25,-2.5) node {{\scalebox{0.8}{$\tau(m)$}}} ; 								
\end{tikzpicture}}}
\end{align*}
\caption{The shifted homotopy Lie bialgebra obtained by Koszul hierarchy.}
		\label{Fig:HoBiLie}
\end{figure*}

\begin{proof}
We claim first that the genus $0$ operations of a shifted homotopy involutive Lie bialgebra forms a shifted homotopy Lie bialgebra. 
This can be viewed from the injective morphism of quadratic data \cite{ManinVallette2020} underlying the surjective morphism of properads 
$\mathrm{sBiLie} \twoheadrightarrow \mathrm{sIBiLie}$, which induces an injective morphism of coproperads 
$\mathrm{sBiLie}^{\ac} \hookrightarrow \mathrm{sIBiLie}^{\ac}$
and then an injective morphism between the Koszul resolutions
$\mathrm{sBiLie}_\infty \hookrightarrow \mathrm{sIBiLie}_\infty$~.
So the genus $0$ operations of the Koszul hierarchy forms a shifted homotopy Lie bialgebra, which is gauge homotopy trivial. 

\medskip 

It remains to compute the genus $0$ coefficients $A^0_{m,l}$ from \cref{thm:KoszulHierarchy}: we claim that 
\[A^{0}_{m,l}=(-1)^{l+m-1} m! m^{l-1}
~,\]
for $l\geqslant 0$ and $m\geqslant 1$~. 
For $l=0$ and $m=1$, we get $A^0_{1,0}=1$~. Higher up, we have 
\begin{align*} 
\sum_{\substack{
0 \leqslant l \leqslant n\\
1\leqslant p\leqslant m}}
\binom{n}{l} 
(-1)^{l+p-1} p! p^{l-1}
\, 
\mathrm{T}(m, n-l, 0,p)&=
\sum_{\substack{
0 \leqslant l \leqslant n\\
1\leqslant p\leqslant m}}
\binom{n}{l} 
(-1)^{l+p-1} (p-1)! p^l
\, 
\mathrm{S}(m, p)p^{n-l}\\
&=
\sum_{p=1}^m
(-1)^{p-1} (p-1)! p^n \mathrm{S}(m, p)
\sum_{l=0}^n
\binom{n}{l} 
(-1)^{l} 
=0~.
\end{align*}
\end{proof}

Restricting even further the structure operations to one output, we recover the original Koszul hierarchy of \cite{Koszul85}, 
see \cite[Section~2.4]{GCTV12}, 
that is the shifted $\mathrm{L}_\infty$-algebra 
\[
\ell_n=\nu_{1,n}=
\sum_{\substack{
k+l=n \\ 
\sigma\in \mathsf{Sh}(k, l)}} 
(-1)^{l} \, 
\left(\left(\mu^{l} \circ_1 \d\right) \circ_1 \mu^{k-1} \right) \cdot \sigma
=
\sum_{\substack{
k+l=n \\ 
\sigma\in \mathsf{Sh}(k, l)}} 
(-1)^{l} \,
\vcenter{\hbox{\begin{tikzpicture}[scale=0.4]
	\draw[thick] (0,0)--(0,1);
	\draw[thick] (-1,2)--(0,1)--(0,2)--(0,1)--(1,2);
	\draw[thick] (0,0)--(1,-1)--(1,0)--(1,-1)--(2,0);		
	\draw[thick] (1,-2)--(1,-1);
	\draw (0,0) node {{$\bullet$}} ; 	
	\draw (-1.25,2.5) node {{\scalebox{0.8}{$\sigma^{-1}(1)$}}} ; 				
	\draw (2.5,0.5) node {{\scalebox{0.8}{$\sigma^{-1}(n)$}}} ; 						
\end{tikzpicture}}}
\]
coming from the data of a chain complex $(A, \d)$ and a commutative algebra structure $(A, \mu)$.  
Restricting the structure operations to just one input, one gets a  new shifted $\mathrm{L}_\infty$-coalgebra 
\begin{align*}
\mathcal{c}_m=\nu_{m,1}&=
(-1)^{m-1}(m-1)!\, 
\sum_{\tau \in \mathsf{Sh}(1, m-1)}
(-1)^{m}(m-2)! \, 
\tau^{-1}\cdot \left( \left(\Delta^{m-2} \circ_1^1\d \right) \circ_1^2 \Delta\right)
\\&=
(-1)^{m-1}(m-1)!\, 
\vcenter{\hbox{\begin{tikzpicture}[scale=0.4]
	\draw[thick] (1,0)--(1,1);
	\draw[thick] (1,-1)--(1,0);		
	\draw[thick] (0,-2)--(1,-1)--(2,-2)--(1,-1)--(0.66,-2)--(1,-1)--(1.33,-2);
	\draw (1,0) node {{$\bullet$}} ; 	
	\draw (1,1.5) node {{\scalebox{0.8}{$1$}}} ; 				
	\draw (0,-2.5) node {{\scalebox{0.8}{$1$}}} ; 							
	\draw (2,-2.5) node {{\scalebox{0.8}{$m$}}} ; 								
\end{tikzpicture}}}
 +
\sum_{\tau \in \mathsf{Sh}(1, m-1)}
(-1)^{m}(m-2)! \, 
\vcenter{\hbox{\begin{tikzpicture}[scale=0.4]
	\draw[thick] (1,0)--(0,1)--(-1,0);
	\draw[thick] (0,1)--(0,2);
	\draw[thick] (1,-1)--(1,0);		
	\draw[thick] (0,-2)--(1,-1)--(2,-2)--(1,-1)--(0.66,-2)--(1,-1)--(1.33,-2);
	\draw (1,0) node {{$\bullet$}} ; 		
	\draw (0,2.5) node {{\scalebox{0.8}{$1$}}} ; 				
	\draw (-1.25,-0.5) node {{\scalebox{0.8}{$\tau(1)$}}} ; 							
	\draw (-0.25,-2.5) node {{\scalebox{0.8}{$\tau(2)$}}} ; 							
	\draw (2.25,-2.5) node {{\scalebox{0.8}{$\tau(m)$}}} ; 								
\end{tikzpicture}}}
\end{align*}
produced by the data of a chain complex $(A, \d)$ and a cocommutative coalgebra structure $(A, \Delta)$.  
Turning these operations upside down, we get a canonical formula for a new shifted $\mathrm{L}_\infty$-algebra produced by a chain complex and a commutative algebra structure. This vertical mirror symmetric shifted $\mathrm{L}_\infty$-algebra is actually the one produced by the action of $\Theta$ on $\d$, and not by the action of $\Theta^{-1}$ as in the definition of the Koszul hierarchy. This other convention is equally interesting and, more generally, the present construction admits a large degree of freedom: 
one can act by any $\Theta^k$, with $k\in \ZZ$, 
one can change the coefficients of the assignment $\mathrm{K}$, 
and, last but not least, one can start from other types of bialgebras like Koszul double Poisson gebras \cite[Theorem~4.17]{LV22} to produce pre-Calabi--Yau algebras \cite{KTV25}.

\section{Twisting procedure}\label{sec:twisting}
In this section, we adapt the arguments of \cite[Section~4.5]{DotsenkoShadrinVallette22}, 
with some slight modifications,  
 in order  to settle the twisting procedure on  the properadic level, whose heuristic is as follows.
We work with a filtered infinitesimal dg coproperad $\oC$ and a complete dg module $A$: they give rise to 
the complete convolution Lie-graph algebra
\[
\g_{\mathcal{C}, A}=\left(\hom_{\Sy}\left(\oC, \eend_A\right), \F, \partial, \{\gra\}_{\gra \in \dcGra}\right)
\]
by \cite[Proposition~1.28]{CV25I}.
Suppose that the filtered infinitesimal coproperad $\oC$ admits an ``extension'' filtered counital infinitesimal coproperad $u\C$
containing an arity $(1,0)$ element $u$ of degree $0$ and such that $u\C$ projects onto $\oC$ under a map $\pi \, \colon \, u\C \twoheadrightarrow \oC$~. 
In this case, the original convolution algebra $\g_{\mathcal C,A}$ embeds into the convolution algebra associated to $u\C$: 
\[\pi^*\ \colon \ \g_{\C, A}=\hom_{\Sy}\big(\oC, \eend_A\big)\hookrightarrow  \g_{u\C, A}=\hom_{\Sy}\big(u\C, \eend_A\big)\ .\]
Under the assignment $u \mapsto a$, this bigger convolution algebra contains the elements of $A$; it also contains 
the two elements  
\[
\begin{tikzcd}[column sep=normal, row sep=tiny]
	\1 \ : \ u\C \ar[r,"\varepsilon"]  & \I \ar[r,"\id_A"] & \eend_A  &\text{and} &
	\delta \ : \ u\C \ar[r,"\varepsilon"]  & \I \ar[r,"\partial_A"] & \eend_A~, 
\end{tikzcd} 
\]
where $\varepsilon $ is the counit of $u\C$. 
Any Maurer--Cartan element $\bar \alpha$ in  $\g_{\mathcal{C},A}$ gives a Maurer--Cartan $\bar{\alpha}\pi$ in $\g_{u\C, A}$, that we prefer to view as $\alpha \coloneq \delta + \bar{\alpha}\pi$: notice that here $\delta$ is not a formal element, so $\alpha$ really denotes a square-zero element of degree $-1$ of $\g_{u\C, A}$.
The deformation gauge group of $\g_{u\C, A}$ is $\1+\F_1 \left(\g_{u\C, A}\right)_0$, so it gives rise to a bigger deformation gauge group, which, for instance, contains the elements $a\in \F_1A_0$ under the form 
$\1+a$. Such elements act on $\alpha$  in $\g_{u\C, A}$ to provide us with elements
\[ (\1 +a)\cdot \alpha \in \MC\left(\g_{u\C, A}\right) \ ,\]
This Maurer--Cartan element in $\g_{u\C, A}$ induces a Maurer--Cartan element in $\g_{\mathcal C,A}$ if and only if it factors through $\pi$: 
\[
\begin{tikzcd}
u\C \ar[rr,"(\1 +a)\cdot \alpha"]  \ar[rd, "\pi"'] & & \eend_A\\
& \oC \ar[ur, dashrightarrow]&
\end{tikzcd}
\]

\medskip

In practice, this properadic twisting procedure can be performed for instance in the following 
way. Given a properadic quadratic data $(E,R)$, that is 
$R\subset \mathcal{G}(E)^{(2)}$, we consider the data $\chi$ of: 
\begin{itemize} 
\item two linear maps $\{\chi_i : E(1,2) \to \k\}_{i=1,2}$, 
\item a certain choice of some basis elements $\left\{e^\chi_{m,n}\right\}$ of 
$E(m,n)$, for any $(m,n)\neq (1,2)$, 
\item for each $e^\chi_{m,n}$ a choice of some inputs $\left\{i^\chi_{m,n}\right\}$.
\end{itemize}
Let $u$ denote an arity $(1,0)$ degree $0$ element.
This gives rise to a space of relations $R_\chi\subset \mathcal{G}(E\oplus \k u)$, 
 made up of 
\[\mu \circ_1^1 u-\chi_1(\mu)\id\quad \text{and}\quad \mu \circ_2^1 u-\chi_2(\mu)\id\ ,\]
with $\mu\in E(1,2)$, and the composites 
\[e\circ_{i}^1 u\]
for any $e\in \left\{ e^\chi_{m,n}\right\}$, 
any $i\in \left\{i^\chi_{m,n}\right\}$, and any $(m,n)\neq (1,2)$. 

\begin{definition}[Unital extension]
The \emph{unital extension} of $\P\coloneq\P(E,R)$ by $\chi$ is the properad defined by
\[u_\chi\P\coloneq
\P\big(E\oplus \k u, R\oplus R_\chi\big)\cong
\frac{\P(E,R) \vee \U}{\left(
R_\chi
\right)}\ ,
\]
where $\vee$ stands for the coproduct of properads and where 
$\U$ stands for the properad spanned $u$~.
\end{definition}

A $u_\chi\P$-gebra is a $\P$-gebra equipped with a degree $0$ element which acts as a unit, possibly with coefficients, for some  generating binary products and such that its composite at the chosen inputs of the chosen other generating operations vanish. 

\begin{definition}[Extendable quadratic data]
A properadic quadratic data $(E,R)$ is \emph{extendable} if there exists a data $\chi$ with non-trivial 
pair of maps $\{\chi_i : E(1,2) \to \k\}_{i=1,2}$
such that the canonical map 
$
\P \hookrightarrow u_\chi\P
$ 
is a monomorphism. 
\end{definition}

For instance, if one chooses to kill all the composites of $u$ with any element of $E(m,n)$, for $(m,n)\neq (1,2)$, the extendability  condition is equivalent to the maximality of the underlying $\Sy$-bimodule of  the unital extension, that is 
$u_\chi \P\cong u \oplus \P$~.

\begin{proposition}\label{lemma:Extendable}\leavevmode
\begin{enumerate}
\item The quadratic properads $\Frob$ and $\mathrm{IFrob}$ of (involutive) Frobenius bialgebras are extendable. 

\item The quadratic properads $\mathrm{BiLie}$ and $\mathrm{IBiLie}$ of (involutive) Lie bialgebras are not extendable.
\end{enumerate}
\end{proposition}

\begin{proof}Recall that the involutivity relation states the genus $1$ composite of the coproduct followed by the product vanishes. 
\begin{enumerate}
\item The cases of $\Frob$ and $\mathrm{IFrob}$ can be treated in the same way as the operad $\mathrm{Com}$ of commutative algebras, see the proof of Proposition~4.21 in \cite{DotsenkoShadrinVallette22}: in these cases, one requires that $u$ is a unit for the binary product and one can require or not the binary coproduct to kill $u$. 

\item 
The two cases  $\mathrm{BiLie}$ and $\mathrm{IBiLie}$ contain the operad $\mathrm{Lie}$ which is shown in \emph{loc. cit.} not to be extendable regardless of any choices we make. 
\end{enumerate}
\end{proof}

In practice, the framework described above appears naturally in the context of the Koszul duality for properads, see \cite{Vallette07} for the homogeneous case. This homological theory was developed in \emph{loc. cit.} for reduced properads, that is having an underlying $\Sy$-bimodule concentrated in arities $(m,n)$ for $m,n \geqslant 1$, but it actually holds as well when their underlying $\Sy$-bimodule is concentrated in arities $(m,n)$ for $m \geqslant 1$ and $n\geqslant 0$, see \cite[Section~3.1]{BV24} for more details. 
Recall that a quadratic properad $\P$ admits a presentation with generating $\Sy$-bimodule $E$ and quadratic relations $R\in \mathcal{G}^{(2)}(E)$ has a \emph{Koszul dual coproperad} 
$\P^{\ac}\coloneq \mathcal{C}\left(sE, s^2 R\right)$
cogenerated by $sE$ with corelations $s^{2}R$.
The cobar construction $\P_\infty\coloneq \Omega \P^{\ac}$ provides us with a cofibrant properad, which is a resolution of $\P$ when this  properad is Koszul. 
The \emph{Koszul dual properad} is defined by the arity-wise and degree-wise linear dual, which admits the following presentation 
\[\P^{!}\coloneq \left(\P^{\ac}\right)^*\cong \P\left(s^{-1}E^*, s^{-2}R^\perp\right)~,\]
when $E$ is arity-wise and degree-wise finite dimensional by \cite[Lemma~1.30]{LV22}.

\medskip
Let us now apply the previous discussion to the Koszul dual properad of a quadratic properad. 

\begin{definition}[Twistable $\P_\infty$-gebras]\label{def:TwHoGebra}
The category of $\P_\infty$-gebras
associated to an arity-wise and degree-wise finite dimensional quadratic presentation $(E,R)$
 is said to be \emph{twistable} if the Koszul dual properad $\P^!$ is extendable and if 
the linear dual  of the extension properad forms an infinitesimal coproperad, that we denote simply by 
\[c\P^{\ac}\coloneq \left(u_\chi \P^!\right)^*~.\]
\end{definition}

\begin{remark}
The last condition of this definition comes from the following key issue: unlike algebras, operads, and even dioperads, the linear dual of an arity-wise and degree-wise finite dimensional properad does not necessarily form a infinitesimal coproperad as one might still encounter infinite sums there. The key point lies in the fact that $2$-vertex graphs can have arbitrary genus. 
A counter-example is given by the properad of Frobenius bialgebras satisfying the extra relation
$\EE=\II$~, which is one-dimensional in each arity $(m,n)$, for $m, n\geqslant 1$. 
\end{remark}

\begin{proposition}\leavevmode
\begin{enumerate}
\item The categories of shifted homotopy Lie bialgebras and shifted homotopy involutive Lie bialgebras are twistable. 

\item The categories of shifted homotopy double Poisson gebras, shifted homotopy infinitesimal balanced bialgebras, and shifted pre-Calabi--Yau algebras are twistable. 
\end{enumerate}
\end{proposition}

\begin{proof}\leavevmode
\begin{enumerate}
\item Their Koszul dual properads $\Frob$ and $\mathrm{IFrob}$ are extendable by \cref{lemma:Extendable} and their explicit descriptions given in \cref{sec:hierarchy} show that their linear duals produce infinitesimal coproperads each time: any basis element can only be written in a finite number of ways as the composite of two others. 

\item The first two cases have already been treated in \cite{LV22} and \cite{Q23} respectively, without the notion of an extendable quadratic data. A unital extension $\mathrm{uDPois}^!$ is defined in \cite[Definition~1.37]{LV22} and the inclusion $\mathrm{DPois}^! \hookrightarrow \mathrm{uDPois}^!$ is a consequence of Lemmata~1.33 and 1.39 of \emph{loc. cit.}. 
Similarly, a unital extension $\mathrm{uBIB}^!$ is defined in \cite[Section~3]{Q23} and the inclusion $\mathrm{BIB}^! \hookrightarrow \mathrm{uBIB}^!$ is the statement of Theorem~22 of \emph{loc. cit.}. In these two cases, 
the linear dual of these extension properads produce infinitesimal coproperads by the same argument as above.
Regarding the last example of shifted pre-Calabi--Yau algebras, one has to consider the quadratic dioperad $\mathrm{V}$ of \cite{TZ07Bis, PT19} and one has to work with the present theory but for dioperads, which works the exact same way. The dioperad $\mathrm{V}$ is Koszul self-dual, see \cite[Proposition~2.7]{LV22}, and adding a unit with respect to the associative product to its Koszul dual produces an extension dioperad, whose linear dual gives a codioperad, that is an infinitesimal coproperad with images of the infinitesimal decomposition maps made up of graphs of genus $0$. 
\end{enumerate}
\end{proof}

\begin{remark}
Modifying the degree condition on $u$ from the beginning, one can remove the shifting assumption in this statement. 
\end{remark}

{\sc Assumption.} {From now on, we will work with a quadratic properad $\P$, whose $\P_\infty$-gebras are twistable and with a 
complete graded module $A$}~.

\begin{definition}[Curved $\P_\infty$-gebras]
We call \emph{curved $\P_\infty$-gebra structures} on $A$ the Maurer--Cartan elements of the extended convolution algebra 
\[\g_{c\P^{\ac}, A}=\hom_{\Sy}\left(c\P^{\ac}, \eend_A\right)~.\]
\end{definition}

\begin{remark}
This notion of a curved $\P_\infty$-gebra is encoded by the cobar construction 
of the extension infinitesimal coproperad $c\P^{\ac}$ defined by 
$\Omega c\P^{\ac} \coloneq \left(\mathcal{G}\left( s^{-1} c\P^{\ac}\right), d_2\right)$~.
\end{remark}

In the cases of the operads $\mathrm{Ass}$ and $\mathrm{Lie}$, one recovers this way 
the notions \emph{curved $\mathrm{A}_\infty$-algebras} and \emph{curved $\mathrm{L}_\infty$-algebras} respectively, see \cite[Chapter~4]{DotsenkoShadrinVallette22}. 
In the case of the dioperad $\mathrm{V}$, one recovers the notion of a \emph{curved pre-Calabi--Yau algebra} of \cite[Section~2]{LV22}. 
In the cases of the properads $\mathrm{DPois}$ and $\mathrm{BIB}$, one recovers the notion of a \emph{curved homotopy double Poisson gebra} of \cite[Section~1.5]{LV22} and the notion of a \emph{curved homotopy infinitesimal balanced gebra} of \cite[Section~3]{Q23}. In these last two cases, there are no relations in the extension of the Koszul dual properad coming from the vanishing of the composite of the unit with a non-binary operation.

\medskip 

\begin{remark}
Applying the integration theory of complete Lie-graph algebras settled in \cite[Section~2]{CV25I}, to the extended convolution algebra, one can define a notion of \emph{$\infty$-morphism}  
$f \colon (A,\alpha) \rightsquigarrow (B, \beta)$ of curved $\P_\infty$-gebras as a degree $0$ map of complete $\Sy$-bimodules $f \colon c\P^{\ac} \to  \eend^A_B$  such that $f(u^*)\in \mathrm{F}_1 B_0$ and  satisfying the equation
$f  \rhd \alpha = \beta \lhd f$~.
They form a category under the \emph{composition} of $\infty$-morphisms defined by $g\circledcirc f$~. 
\end{remark}

\begin{definition}[Twisting procedure]
Given a curved $\P_\infty$-gebra 
$\alpha \in \MC\left(\g_{c\P^{\ac}, A}\right)$, the action on $\alpha$ 
of the gauge element $\1-a$, associated to any element $a\in \mathrm{F}_1 A_0$, 
 produces a 
curved $\P_\infty$-gebra 
\[ (\1 -a)\cdot \alpha = \alpha \lhd(\1 +a)  \in \MC\left(\g_{c\P^{\ac}, A}\right) \ ,\]
called 
\emph{twisted by $a$}, and denoted simply by $\alpha^a$~. 
\end{definition}

This formula is a consequence of \cite[Theorem~2.29]{CV25I}; it amounts to considering all the possible ways to plug $a$ or the identity at the inputs of the operations of the curved $\P_\infty$-gebra $\alpha$. This recovers the twisting procedure of \cite[Theorem~4.20]{LV22} for curved homotopy double Poisson gebras, curved pre-Calabi--Yau algebras, and curved homotopy infinitesimal balanced bialgebras by \cite[Theorem~14]{Q23}. 

\begin{remark}
Under the aforementioned terminology, $\1-a$ is an $\infty$-isotopy between $\alpha$ and the twisted curved $\P_\infty$-gebra $\alpha^a$.
\end{remark}

\begin{proposition}
Any curved $\P_\infty$-gebra twisted by an element $a$ and then by an element $b$ is equal to the curved $\P_\infty$-gebra twisted by the element $a+b$, that is $\left(\alpha^a\right)^b=\alpha^{a+b}$~. 
\end{proposition}

\begin{proof}
This a direct corollary of the fact that we are dealing with a group action whose product satisfies
\[(\1 -b)\circledcirc (\1-a)=\1 -(a+b)~,\]
by \cite[Theorem~2.6]{CV25I}.
\end{proof}

The issue is now to settle a twisting procedure for $\P_\infty$-gebras. The inclusion of properads 
$\P^! \hookrightarrow u_\chi\P^!$ induces the projection of infinitesimal coproperads 
$c\P^{\ac} \twoheadrightarrow \overline{\P}^{\ac}$ and thus the inclusion of complete Lie-graph algebras
$
\g_{\P^{\ac}, A} \hookrightarrow
\mathfrak{g}_{c\P^{\ac}, A}$~. 
So any $\P_\infty$-gebra structure on $A$ can be viewed as a curved $\P_\infty$-gebra structure that can thus be twisted. 
But the result will \emph{a priori} produce a curved $\P_\infty$-gebra. It remains to measure the difference between these two notions. 
To this extent, we consider the $\Sy$-bimodule $\mathcal{K}$ which is the kernel of the projection map: 
\[\mathcal{K} \hookrightarrow c\P^{\ac} \twoheadrightarrow {\P}^{\ac}~.\]

Putting all the above arguments together, we finally get the twisting procedure for $\P_\infty$-gebras. 

\begin{theorem}\label{thm:TwProp} 
For any complete $\P_\infty$-gebra structure $\bar\alpha \in \MC(\g_{\P^{\ac},A})$ on $A$ and for any 
element  $a\in \F_1A_0$, 
 the Maurer--Cartan element 
 \[(\1 -a)\cdot \alpha = \alpha \lhd(\1 +a)  \in \MC\left(\g_{c\P^{\ac}, A}\right)
 \] 
 of the extended convolution algebra is a $\P_\infty$-gebra structure on $A$ if and only if its 
 image on $\mathcal{K}$ vanishes: 
 \[\big(\alpha \lhd(\1 +a)\big)(\mathcal{K})=0~.\]
 \end{theorem}

\begin{proof}
The proof is a straightforward consequence of the aforementioned arguments.
\end{proof}

This conceptual approach motives the following general definition. 

\begin{definition}[$\P_\infty$-Maurer--Cartan equations]\label{def:PinftyMCequation}
The equations 
 \[\big(\alpha \lhd(\1 +a)\big)(k)=0~,\]
running over a basis $\{k\}$ of $\mathcal{K}$,  are called the \emph{$\P_\infty$-Maurer--Cartan equations}.
An element $a\in \F_1A_0$ satisfying them is called a \emph{Maurer--Cartan element} of the $\P_\infty$-gebra $\alpha$. 
\end{definition}

In the case of the operads $\mathrm{Ass}$ and $\mathrm{Lie}$, the kernel $\mathcal{K}$ is the one-dimensional module spanned by $u^*$.
This way, one recovers conceptually the usual Maurer--Cartan equations 
\[da+\mu_2(a,a)+\mu_3(a,a,a)+ \cdots =0
\qquad \text{and} \qquad 
da+\tfrac12\ell_2(a,a)+\tfrac{1}{3!}\ell_3(a,a,a)+ \cdots =0
\]
of $\mathrm{A}_\infty$-algebras and $\mathrm{L}_\infty$-algebras respectively.  The Maurer--Cartan equation for pre-Calabi--Yau algebras is explicitly given in \cite[Definition~4.21]{LV22}, where it is shown to coincide with the Maurer--Cartan equation of the underlying 
$\mathrm{A}_\infty$-algebra.

\medskip

Let us now treat in detail the twisting procedure for shifted homotopy involutive Lie bialgebras,  using the notations introduced in \cref{sec:hierarchy}.  
First, the Koszul dual properad $\mathrm{sIBiLie}^!\cong\Frob$ is extendable: we consider its extension given by 
the properad $\uFrob$ encoding 
\emph{unital Frobenius bialgebras} where one unit for the commutative product is added and no extra relation between the unit and the coproduct is required. Its linear dual infinitesimal coproperad $\uFrob^*$ admits for basis the elements 
$\left\{c_{m,n}^g\right\}_{m \geqslant 1, n, g \geqslant 0}$, which extend the basis of $\Frob^*$ to $n=0$~. 
So the complete convolution Lie-graph algebra encoding 
shifted curved homotopy involutive Lie bialgebras is the quantum Weyl algebra
\[\g_{{\mathcal \uFrob^*}, A}\cong 
\prod_{\substack{m \geqslant 1 \\ n,g  \geqslant 0}} \hom\left(A^{\hat{\odot} n}, A^{\hat{\odot} m}\right) \hbar^g\]
and the kernel $\mathcal{K}$ of the canonical projection $\uFrob^* \twoheadrightarrow \Frob^*$ is the $\Sy$-bimodule 
concentrated in input arity $0$ with
$\mathcal{K}(m,0)\cong\k\left\{c_{m,0}^g\, , \ g\geqslant 0\right\}$~. Given a shifted homotopy involutive Lie bialgebra structure
$\left\{\nu_{m,n}^g\right\}_{m, n \geqslant 1, g \geqslant 0}$ on a complete graded module $A$, its Maurer--Cartan elements are the  elements 
$a\in \F_1A_0$ satisfying the Maurer--Cartan equations: 
\[
\sum_{n\geqslant 1} \tfrac{1}{n!} \nu^g_{m,n}(a, \ldots, a)=0
\]
for $m\geqslant 1$ and $g\geqslant 0$. 
(Notice that for $m=1$ and $g=0$, we recover the Maurer--Cartan equation of the shifted homotopy Lie sub-algebra.) 
In this case, the formulas 
\[
\left(\nu_{m,n}^g\right)^a=\sum_{k\geqslant 0}
\tfrac{1}{k!}
\nu_{m,k+n}^g(\underbrace{a, \ldots, a}_{k \ \text{times}}, -, \ldots, -)
\]
define the shifted homotopy involutive Lie bialgebra twisted by $a$. 
For shifted homotopy Lie bialgebras, that is shifted homotopy involutive Lie bialgebras concentrated in $g=0$, this recovers the 
twisting procedure described in \cite[Appendix~B]{MW15}, which is shown there to admit a nice application to the cohomology classes of the moduli spaces of curves with marked points. 

\medskip

This is the special case of \cite[Proposition~9.3]{CFL15} when $\mathfrak{m}_{l,g}=0$, except for $l=1$ and $g=0$ where 
$\mathfrak{m}_{1,0}=a$.
It is worth noticing that the present theory includes the full version of this result as follows. The entire space 
\[\hom_{\Sy}\left( \mathcal{K}, \eend_A\right)\cong \prod_{\substack{m \geqslant 1 \\ g  \geqslant 0}}  A^{\hat{\odot} m} \hbar^g\]
is a direct summand of the complete convolution Lie-graph algebra $\g_{{\mathcal \uFrob^*}, A}$ whose elements 
\[\left\{a_{m}^g\in \allowbreak \F_1 \left(A^{\hat{\odot} m}\right)_0\right\}_{m\geqslant 1, g\geqslant 0}\]  can be used to twist any shifted (curved) homotopy involutive Lie bialgebra $\alpha$ using the deformation gauge group action: 
\[\left(\1 - \sum_{m\geqslant 1 \atop g\geqslant 0}a_{m}^g\right)\cdot \alpha = \alpha \lhd\left(\1 +\sum_{m\geqslant 1 \atop g\geqslant 0}a_{m}^g\right)~.\]
Such a formula amounts to plugging the elements $\left\{a_{m}^g\right\}_{m\geqslant 1, g\geqslant 0}$ and the identity at the inputs of $\left\{\nu_{m,n}^g\right\}_{m, n \geqslant 1, g \geqslant 0}$ in all the possible ways, which coincides with the formula of \emph{loc. cit.}. 

\begin{remark}
The procedure described here twists homotopy gebra structures with elements $a\in A$ but it would work as well with ``coelements'' 
$f \in A^*$: all the arguments are similar under the upside-down (input-output) mirror symmetry. 
Since the properads $\mathrm{BiLie}\text{-}\mathrm{Frob}_\diamond$ and $\mathrm{BiLie}_\diamond\text{-}\mathrm{Frob}$ are symmetric with respect to the 
upside-down mirror symmetry, the formulas for the twisting procedure with a coelement are the same, up to this symmetry, than the formulas for the twisting with an element.
For instance, the notion of a ``cocurved'' Lie bialgebra can already be found in \cite[Definition~2.24]{Safronov2018}.
\end{remark}

All these arguments show that, like the Koszul hierarchy, the fundamental twisting procedure established here 
by means of deformation gauge group action
has many degrees of freedom and a wide range of possible applications.

\bibliographystyle{alpha}
\bibliography{bib}

\end{document}